\documentclass[12pt,reqno]{amsart}
\usepackage{amssymb,mathrsfs,graphicx}
\usepackage{mathtools}
\usepackage{pifont,color}
\usepackage{enumitem}
\usepackage[margin=3cm]{geometry}
\usepackage{tikz}

\newtheorem{theorem}{Theorem}[section]
\newtheorem{lemma}[theorem]{Lemma}

\newtheorem{proposition}[theorem]{Proposition}

\theoremstyle{remark}
\newtheorem{remark}{Remark}[section]

\def\pa{\partial}
\def\rhob{\widetilde{\rho}}
\def\mrhob{{-\widetilde{\rho}}}
\def\R{\mathbb{R}}
\def\rhoM{\rho_M}
\def\rhoc{\rho_c}
\def\d{\mathrm{d}}

\def\eps{\varepsilon}

\allowdisplaybreaks
\begin{document}

\title[Nonlocal ARZ models]{
  On the Aw-Rascle-Zhang traffic models with nonlocal look-ahead interactions}

\author[Thomas Hamori]{Thomas Hamori}
\address[Thomas Hamori]{\newline Department of Mathematics, \ 
 University of South Carolina, Columbia, SC 29208, USA}
\email{thamori@email.sc.edu}

\author[Changhui Tan]{Changhui Tan}
\address[Changhui Tan]{\newline Department of Mathematics, \ 
 University of South Carolina, Columbia, SC 29208, USA}
\email{tan@math.sc.edu}

\date{\today}
\subjclass[2010]{76A30, 35L65, 35B65, 35B51.}
\keywords{nonlocal conservation law, traffic flow, second-order model, global regularity, critical threshold}
\thanks{\textit{Acknowledgment.} This work has been supported by the NSF grants DMS-2108264 and DMS-2238219.}

\begin{abstract}
We present a new family of second-order traffic flow models, extending the Aw-Rascle-Zhang (ARZ) model to incorporate nonlocal interactions. Our model includes a specific nonlocal Arrhenius-type look-ahead slowdown factor. We establish both local and global well-posedness theories for these nonlocal ARZ models.

In contrast to the local ARZ model, where generic smooth initial data typically lead to finite-time shock formation, we show that our nonlocal ARZ model exhibits global regularity for a class of smooth subcritical initial data. Our result highlights the potential of nonlocal interactions to mitigate shock formations in second-order traffic flow models.

Our analytical approach relies on investigating phase plane dynamics. We introduce a novel comparison principle based on a mediant inequality to effectively handle the nonlocal information inherent in our model.
\end{abstract}

\maketitle

\section{Introduction}\label{sec:intro}

\subsection{Macroscopic traffic flow models}

For the past century, the study of traffic flow has been of continued interest to mathematicians investigating the dynamics of vehicles on roadways using a variety of models. One celebrated framework is that of first-order macroscopic traffic flow models, in which one describes the evolution of the traffic density
\begin{equation}\label{eq:basic}
 \pa_t \rho + \pa_x\big(\rho u\big)=0, \quad u=U(\rho).
\end{equation}
Here, the velocity $u$ depends on the local density $\rho$. The speed-density relationship function $U$ is decreasing in $\rho$. Through a scaling argument, we may assume without loss of generality that the maximum density $\rho_{\max}=1$ and the maximum velocity $u_{\max}=1$. Hence, the general assumptions on $U$ reads:
\begin{equation}\label{eq:Ucond}
U(0) = u_{\max}=1, \quad U(1) = U(\rho_{\text{max}}) = 0, \quad\text{and}\quad U'(\rho)<0,\quad\forall\rho \in (0,1).
\end{equation}

One classical choice for the function $U$ proposed by Greenshields \cite{greenshields1935study} is the linear relation
\begin{equation}\label{eq:Ulinear}
U(\rho)=1-\rho.	
\end{equation}
The corresponding dynamics \eqref{eq:basic} is known as the Lighthill-Witham-Richards (LWR) model \cite{lighthill1955kinematic, richards1956shock}.
The LWR model effectively captures various real-world traffic phenomena, including shock formations. 

Equation \eqref{eq:basic} can be interpreted as a scalar conservation law:
\[
\pa_t\rho + \pa_x\big(f(\rho)\big)=0, \quad f(\rho)=\rho\, U(\rho).
\]
Here, $f$ is referred to as the \emph{flux} function.
Classical theory on hyperbolic conservation laws provides a robust well-posedness theory for such systems. See e.g. the book of Dafermos \cite{dafermos2016hyperbolic}. In particular, any generic smooth initial data will inevitably lead to shock formations within finite time.

The LWR model exhibits several notable deficiencies. First of all, the flux function 
\begin{equation}\label{eq:fluxLWR}
f(\rho)=\rho(1-\rho)	
\end{equation}
is concave and symmetric in $[0,1]$. However, empirical observations suggest that the flux function (often referred to as the fundamental diagram) tends to be concave in the low-density regime and convex in the high-density regime, see e.g. \cite{del1995functional, li2011fundamental}. To better align with empirical data, we consider a family of concave-convex flux functions:
\begin{equation}\label{eq:fluxcc}
	f''(\rho)<0,\,\, \text{for}\,\,\rho\in[0,\rho_c),\quad\text{and}\quad
	f''(\rho)>0,\,\, \text{for}\,\,\rho\in(\rho_c,1),
\end{equation}
where $\rho_c$ is the inflection point that distinguishes the two regimes. Note that concave fluxes can be accommodated within this family \eqref{eq:fluxcc} by setting $\rho_c=1$.

One prototypical example of flux functions belonging to the family \eqref{eq:fluxcc}, as proposed by Pipes \cite{pipes1966car}, is given by:
\begin{equation}\label{eq:fluxJ}
    f(\rho) = \rho(1-\rho)^J,\quad J\geq 1,
\end{equation}
where the velocity function is nonlinear:
\begin{equation}\label{eq:Unonlinear}
	U(\rho) = (1-\rho)^J.
\end{equation}
When the parameter $J=1$, it resembles the flux in the LWR model \eqref{eq:fluxLWR}. When $J>1$, the flux is concave-convex, with $\rho_c=\frac{2}{J+1}$. This family of fluxes has recently been derived from microscopic cellular automata models in \cite{sun2020class}.

Another deficiency of the LWR model lies in its unrealistic assumption that drivers can instantaneously adjust their velocity. This limitation prompted the exploration of second-order models, where the velocity $u$ has its own dynamics. These models, often referred to as \emph{non-equilibrium} models, allow for the velocity $u$ to deviate from the equilibrium velocity $U(\rho)$.
Quintessential examples of second-order models include the Payne-Witham (PW) model \cite{payne1971model,whitham2011linear}, and the Aw-Rascle-Zhang (ARZ) model \cite{aw2000resurrection,zhang2002non}.

Our primary focus lies with the ARZ model. Aw and Rascle \cite{aw2000resurrection} introduced the following second-order model:
\begin{equation}\label{eq:AR}
    \begin{cases}
     \pa_t\rho+\pa_x(\rho u) = 0, \\    
    \pa_t(u+p(\rho))+u\pa_x(u+p(\rho))=0.
    \end{cases}
\end{equation}
In this model, the function $p(\rho)$ represents the \emph{traffic pressure}, which can also be interpreted as an anticipation factor or hesitation function. This function $p$ is an increasing function of $\rho$. A typical example of the pressure function is given by the $\gamma$-law: $p(\rho)=\rho^\gamma$.

Independently, Zhang \cite{zhang2002non} introduced a second-order model
\begin{equation}\label{eq:Z}
    \begin{cases}
     \pa_t\rho+\pa_x(\rho u) = 0, \\    
     \pa_tu + (u + \rho U'(\rho))\pa_x u = 0.
    \end{cases}
\end{equation}
This model represents an improvement over the PW model, addressing the unrealistic phenomenon where traffic waves travel faster than individual vehicles.

Remarkably, the two systems \eqref{eq:AR} and \eqref{eq:Z} are equivalent, under the relation:
\begin{equation}\label{eq:pUrelation}
p(\rho)=1-U(\rho).	
\end{equation}
In particular, the isothermal pressure $p(\rho)=\rho$ corresponds to the linear relation \eqref{eq:Ulinear}.

It is convenient for us to introduce an auxiliary quantity $\psi$, defined as:
\[ \psi =  u-U(\rho),\]
which represents the deviation of the velocity $u$ from the equilibrium $U(\rho)$. It can be readily verified from \eqref{eq:AR} that $\psi$ is transported by the velocity $u$. Consequently, we can represent the ARZ model \eqref{eq:AR} or \eqref{eq:Z} as follows:
\begin{equation}\label{eq:ARZ}
    \begin{cases}
     \pa_t\rho+\pa_x(\rho u) = 0, \\    
     \pa_t\psi + u \pa_x \psi = 0,\\
     u = \psi+U(\rho).
    \end{cases}
\end{equation}

An advantageous property of the ARZ model is its consistency with the first-order model \eqref{eq:basic}. Indeed, if we take the initial data $\psi_0\equiv0$, then $\psi$ remains zero for all time, and the ARZ model \eqref{eq:ARZ} reduces to the first-order model \eqref{eq:basic}. For general initial data, For general initial data, one may introduce a relaxation term to the dynamics of $\psi$:
\[\pa_t\psi+u\pa_x\psi=-\tfrac{1}{\tau}\psi,\]
where the scaling parameter $\tau>0$ serves as the relaxation coefficient. With the relaxation term, $\psi$ vanishes as $t\to\infty$. Consequently, the velocity $u$ converges to the equilibrium $U(\rho)$ in the long run. Therefore, the asymptotic behavior of the second-order model matches that of the first-order model \eqref{eq:basic}.

The ARZ model \eqref{eq:ARZ}, whether with or without relaxation, has been extensively analyzed in recent literature; see, for instance, \cite{li2003global, lebacque2007aw, garavello2006traffic, yu2019traffic, chaudhuri2023analysis}. Similar to the first-order model, solutions of the ARZ model exhibit shock formations in finite time.

\subsection{Nonlocal traffic flow models}
An emerging domain of exploration in mathematical models for traffic flow involves the assimilation of nonlocal effects, inspired by the growing availability of nonlocal information to drivers (e.g. smartphone map systems that report downstream traffic conditions). These models operate under the assumption that drivers, whether autonomous or not, adjust their speed based on downstream traffic information. The mechanism by which this effect is constructed varies within the literature. A common approach is to introduce the system with the nonlocal quantity 
\begin{equation*}\label{eq:rhobar}
\rhob= \int_0^\infty w(z) \rho(x+z)dz,
\end{equation*}
which represents nonlocal traffic mass with a weight $w$. Typical assumptions regarding the weight function $w$ include:
\begin{equation}\label{eq:w}
	\quad w(z)\geq0,\quad w(z) \text{ is bounded and non-increasing on }[0,\infty). 
\end{equation}
Thus, drivers react to the traffic conditions ahead, with closer traffic being weighted more heavily.

A class of first-order nonlocal traffic flow models is introduced in the form:
\begin{equation}\label{eq:basicn}
\pa_t\rho+\pa_x(\rho u) = 0, \quad u = \widetilde{U}(\rho,\rhob),
\end{equation}
where the velocity function $\widetilde{U}$ depends nonlocally on the density $\rho$. This system \eqref{eq:basicn} has been extensively studied within the framework of nonlocal conservation laws, including research on well-posedness theory \cite{blandin2016well,goatin2016well,keimer2017existence, keimer2018nonlocal, chiarello2018global, huang2022stability}, asymptotic behaviors \cite{ridder2019traveling, shen2019traveling}, the finite time shock formations \cite{li2011shock,lee2015thresholds, lee2019thresholds, lee2020wave}, and connections to microscopic models \cite{di2019deterministic,sun2023accelerated}.

One interesting choice of the velocity function is
\[\widetilde{U}(\rho,\rhob)=U(\rhob).\]
Numerous studies have focused on the singular limit from the nonlocal equation \eqref{eq:basicn} to the local equation \eqref{eq:basic} by selecting a sequence of weights that approximate the Dirac delta function, such as $w_\eps = \eps^{-1}e^{-s/\eps}$. For further insights, refer to works such as \cite{bressan2020traffic,bressan2021entropy,coclite2022general, colombo2023nonlocal}.

Another notable choice for the velocity function, proposed by Sopasakis and Katsoulakis \cite{sopasakis2006stochastic}, and of particular interest to us, is given by:
\begin{equation}\label{eq:lookahead}
\widetilde{U}(\rho,\rhob)=U(\rho)e^{-\rhob}.	
\end{equation}
In this formulation, the nonlocal input takes the form of an Arrhenius-type look-ahead slowdown factor, representing the concept that heavy traffic ahead leads to a decrease in speed. The system \eqref{eq:basicn} with \eqref{eq:lookahead} has been analyzed in \cite{lee2022sharp, hamori2023sharp}. A significant discovery arising from these studies is that nonlocal look-ahead interactions help prevent the formation of shocks, and subcritical initial data lead to globally smooth solutions. The result suggests that nonlocal interactions can play a crucial role in mitigating the occurrence of traffic congestion.

The understanding of second-order nonlocal traffic flow models is relatively less developed, but it has attracted increasing attention in recent years. A notable contribution is the work by Chiarello et al. in \cite{chiarello2020micro}, where they introduced a generalized ARZ model by incorporating nonlocal dependence on downstream velocity. In their formulation, the third equation in the ARZ model \eqref{eq:ARZ} is replaced by:
\[u(t,x) = \int_\R w(z) \Big(\psi(t,x+z)+U(\rho(t,x+z))\Big)\,dz,\]
with a continuous weight function $w$ satisfying \eqref{eq:w}. They also derive their nonlocal ARZ model from microscopic follow-the-leader dynamics.

\subsection{An ARZ model with nonlocal look-ahead interactions}
In this paper, we propose a new class of second-order nonlocal traffic models. The system is given by:
\begin{equation}\label{eq:NARZ}
    \begin{cases}
     \pa_t\rho+\pa_x(\rho u) = 0, \\    
     \pa_t\psi + u \pa_x \psi = 0,\\
     u = \psi + \widetilde{U}(\rho,\rhob).
    \end{cases}
\end{equation}
This model presents a nonlocal variant of the ARZ model \eqref{eq:ARZ}. We introduce  nonlocality in the third equation, by replacing $U(\rho)$ with the nonlocal velocity function $\widetilde{U}(\rho,\rhob)$.

A notable distinction between our model and the one presented in \cite{chiarello2020micro} is that our model incorporates nonlocal information on the downstream density. Such information is considerably easier to measure compared to the downstream velocity required in \cite{chiarello2020micro}.

Another advantage of our model is its consistency with the first-order nonlocal traffic model \eqref{eq:basicn}. Similar to the local ARZ model, setting $\psi_0\equiv0$ will reduce \eqref{eq:NARZ} to \eqref{eq:basicn}. Hence, some analytical features of the first-order model may be extended to the second-order model.

Our main goal is to establish a global well-posedness theory for our nonlocal ARZ model \eqref{eq:NARZ}. Unlike first-order models, which are scalar conservation laws, the well-posedness theory for second-order models, being \emph{systems} of conservation laws, is considerably less understood. It is widely recognized that entropic weak solutions are not necessarily unique, even for local models. Uniqueness results for nonlocal systems of conservation laws are limited, see e.g. \cite{brenier2013sticky, nguyen2015one, leslie2021sticky}.

In this paper, we study the well-posedness theory for \emph{smooth} solutions, with a particular focus on analyzing the look-ahead slowdown interaction \eqref{eq:lookahead} and understanding its regularization effect, similar to first-order models. We identify a class of subcritical initial data, and obtain global smooth solutions to the initial value problem. 

The major analytical tool we use is the phase plane analysis along characteristic paths. This approach captures the \emph{critical threshold phenomenon}: subcritical initial data lead to global regularity, while supercritical initial data result in finite-time shock formation. This method has proven successful in establishing global well-posedness for numerous nonlocal conservation laws \cite{liu2001critical,tadmor2003critical,tadmor2014critical,tan2020euler, tan2021eulerian, bhatnagar2023critical}, including the first-order nonlocal traffic flow models \cite{lee2015thresholds,lee2019thresholds,lee2020wave,lee2022sharp,hamori2023sharp}.

The structure of our second-order nonlocal model \eqref{eq:NARZ} presents additional challenges in the phase plane analysis. In particular, we need to establish appropriate \emph{comparison principles} to handle the nonlocality. To this end, we introduce and apply a \emph{mediant inequality} in the construction of the threshold condition. This ensures global regularity of the solution to \eqref{eq:NARZ} for a non-trivial family of subcritical initial data. Given that solutions to the local ARZ model exhibit finite-time shock discontinuity for any generic initial data, our result highlights the regularization effect of the nonlocal look-ahead interaction \eqref{eq:lookahead} on the second-order ARZ model.

\subsection{Outline of the paper}
The remainder of this paper is organized as follows. 
In Section \ref{sec:results}, we state our main results and offer some interpretation and discussion, with detailed explanations provided in subsequent sections. 
In Section \ref{sec:lwp}, we establish a local well-posedness theory for \eqref{eq:NARZ}, and a regularity criterion to ensure global well-posedness. We demonstrate that this criterion is not met for the local ARZ model \eqref{eq:ARZ} with generic initial data. 
In Section \ref{sec:gwp}, we show that with the nonlocal look-ahead interaction, the system \eqref{eq:NARZ} is \emph{globally} well-posed for a non-trivial class of subcritical initial data. We give a explicit description of the subcritical region.

\subsection*{Notation}
We denote $L^p(\R)$ the Lebesgue spaces in $\R$, and $H^k(\R)$ the Sobolev space, endowed with the norm
\[\| g\| _{H^k}^2=\| g\| _{L^2}^2+\| \tfrac{d^k}{dx^k}g\| _{L^2}^2,\]
for any non-negative integer $k$. For general $k\geq0$, the space $H^k(\R)$ is defined via Fourier transform
\[\| g\| _{H^k} = \| (I-\Delta)^{k/2}g\| _{L^2} = \big\|\mathcal{F}^{-1}\Bigl[(1+| \xi| ^2)^{k/2}\mathcal{F}g \Bigr]\big\| _{L^2},\]
where $\mathcal{F}$ and $\mathcal{F}^{-1}$ are the forward and inverse Fourier transforms in $\R$, respectively.
We denote $\| \cdot\| _{\dot{H}^k(\R)}$ the homogeneous semi-norm with
\[\| g\| _{\dot{H}^k} = \|\Lambda^kg\|_{L^2}=\| (-\Delta)^{k/2}g\| _{L^2} = \big\| \mathcal{F}^{-1}\Bigl[| \xi| ^k\mathcal{F}g \Bigr]\big\| _{L^2},
\]
where the pseudo-differential operator $\Lambda=(-\Delta)^{1/2}$.

The notations $\lfloor k \rfloor$ and $\lceil k \rceil$ refer to the largest integer less than or equal to $k$, and the smallest integer greater than or equal to $k$, respectively.

We will repeatedly use the letter $C$ to refer to a constant $C>0$ whose value may change line by line. The constant might depend on parameters and initial conditions. We write $C(p)$ to represent that the constant depends on the parameter $p$. 

Finally, we denote $g'$ the derivative of $g$, if $g$ has a single variable, and $\dot g$ denotes the material derivative of $g=g(t,x)$ along the characteristic path $X(t,x)$ satisfying:
\begin{equation}\label{eq:path}
\pa_t X(t,x) = \Big(u+\rho U'(\rho) e^\mrhob\Big)(t,X(t,x)),\quad X(t=0,x) = x,
\end{equation}
so that
\[
\dot g(t,X(t,x)) = \frac{d}{dt}g(t,X(t,x)) = \Big(\partial_t g + \big(u+\rho U'(\rho) e^\mrhob\big)\partial_x g\Big)(t,X(t,x)).
\]

\section{Statement of Main Results}\label{sec:results}

The primary objective of this paper is to investigate the analytical properties of our proposed nonlocal ARZ model:
\begin{equation}\label{eq:main}
 \begin{cases}
  \pa_t\rho + \pa_x(\rho u)  = 0, \\    
  \pa_t\psi + u\pa_x\psi = 0,\\
  u = \psi+U(\rho)e^\mrhob, \quad \rhob(t,x) = \displaystyle\int_\R w(z) \rho(t,x+z)dz,
 \end{cases}
\end{equation}
with initial data
\begin{equation}\label{eq:init}
 \rho(0,x) = \rho_0(x), \quad u(0,x)=u_0(x),\quad   
   \psi(0,x) = \psi_0(x):=u_0(x)-v(\rho_0(x))e^{\mrhob(0,x)}.	
\end{equation}
Throughout the paper, we make following global assumptions to the initial data:
\begin{enumerate}[label=(A\theenumi),ref=A\theenumi]
 \item \label{A1} Finite total mass: $m:=\int_\R\rho_0(x) dx<\infty$.\smallskip
 \item \label{A2} Boundedness: $0\leq\rho_0(x)\leq 1$, $0\leq u_0(x)\leq 1$, for any $x\in\R$.\smallskip
 \item \label{A3} Smoothness: $\rho_0'\in H^{k-1}(\R)$, $u_0'\in H^{k-1}(\R)$, for some large $k$.\smallskip
 \item \label{A4} Monotonicity for $\psi_0$: $\psi_0'(x)\geq0$ for any $x\in\R$.
 \end{enumerate}
 
Several comments on these assumptions are as follows. 

\begin{itemize}
 \item From \eqref{A1} and \eqref{A2}, we know $\rho_0\in L^1\cap L^\infty(\R)$. Hence, $\rho_0\in L^2(\R)$. Together with the smoothness assumption \eqref{A3}, we have
 \[\rho_0\in H^k(\R).\]
 The smoothness assumption on $u_0$ in \eqref{A3} is weaker than $u_0\in H^k(\R)$. In particular, $u_0$ and $\psi_0$ do not necessarily vanish at infinity.
 \item The technical assumption \eqref{A4} plays a crucial role in our analysis. Without this assumption, it is not guaranteed that assumption \eqref{A2} will propagate in time. Specifically, it is known that the solution to the local ARZ model \eqref{eq:ARZ} may violate the maximum principle $\rho(t,x)\leq1$. This violation immediately makes the system ill-posed, as $U(\rho)$ is not physically defined for $\rho>1$. We will demonstrate that the loss of the maximum principle does not occur under assumption \eqref{A4}. For a more detailed discussion, refer to Remark \ref{rmk:psi}.
\end{itemize}

Our well-posedness theory concerns smooth solutions to \eqref{eq:main}, such that $(\rho(t),u(t))$ satisfy \eqref{A1}-\eqref{A4}.

We start with a local well-posedness result. We assume that the velocity function $U$ satisfies the general assumption \eqref{eq:Ucond}. Additionally, in order to achieve the propagation of \eqref{A3}, we require $U$ to be sufficiently smooth, satisfying
\begin{equation}\label{eq:Usmooth}
	U \in C^{\lceil k\rceil+1}([0,1]),\quad\text{and}\quad U^{-1}\in C^{\lceil k\rceil}([0,1]).
\end{equation}
Moreover, we enforce throughout that the weight function $w$ satisfies the look-ahead assumption \eqref{eq:w}.

\begin{theorem}[Local well-posedness]\label{thm:lwp}
  Let $k>\frac32$. Consider the Cauchy problem \eqref{eq:main}-\eqref{eq:init} where the initial data $(\rho_0,u_0)$ satisfy the assumptions \eqref{A1}-\eqref{A4}. Assume $U$ satisfies \eqref{eq:Ucond} and \eqref{eq:Usmooth}, and $w$ satisfies \eqref{eq:w}.
  Then 
  \begin{itemize}
  \item	There exists a time $T>0$ such that the system has a smooth solution $(\rho,u)$ in $[0,T]$, where $(\rho(t), u(t))$ satisfy \eqref{A1}-\eqref{A4} for any $t\in[0,T]$. In particular,
  \[\rho\in C([0,T], H^k(\R)),\quad \pa_xu\in C([0,T], H^{k-1}(\R)).\]
  \item The time $T$ can be extended as long as
\begin{equation}
\label{eq:condition}
\int_0^T \|\pa_x\rho(t,\cdot)\|_{L^\infty}\,dt <\infty .
\end{equation}
  \end{itemize}
\end{theorem}

\begin{remark}\label{rmk:U}
The regularity assumption \eqref{eq:Usmooth} is satisfied when the velocity function is linear \eqref{eq:Ulinear}. However, for nonlinear velocity functions \eqref{eq:Unonlinear} with  $J>1$, the assumption fails at $\rho=1$. To ensure compatibility with the family of fluxes in \eqref{eq:fluxJ} that we are concerned with, we can modify assumption \eqref{A2} as follows:
\begin{enumerate}[label=(A2$'$),ref=A2$'$]
 \item \label{A2p} Boundedness: $0\leq\rho_0(x)\leq \rhoM$, $0\leq u_0(x)\leq 1$, for any $x\in\R$,
\end{enumerate}
where $\rhoM<1$. We will show in Proposition \ref{prop:MP2} that \eqref{A2p} is satisfied for any $t\in[0,T]$.

When \eqref{A2p} holds, the assumption \eqref{eq:Usmooth} can be relaxed to:
\begin{equation}\label{eq:Usmoothp}
	U \in C^{\lceil k\rceil+1}([0,\rhoM]),\quad\text{and}\quad U^{-1}\in C^{\lceil k\rceil}([U(\rhoM),1]).
\end{equation}
Since $\rho$ stays away from 1, the nonlinear velocity functions in \eqref{eq:Unonlinear} satisfy \eqref{eq:Usmoothp}.
\end{remark}

The criterion \eqref{eq:condition} provides a sufficient condition that ensures regularity of the solution. Smooth solutions persist as long as 	$\pa_x\rho$ remain bounded. The only potential occurrence of blow-up is the finite-time shock formation.

However, achieving global well-posedness is challenging. In fact, for the local ARZ model \eqref{eq:ARZ} (or equivalently interpreted as \eqref{eq:main}-\eqref{eq:init} with $w\equiv0$), generic initial data result in finite-time blowup. We state the following theorem for ARZ model with linear velocity function $U(\rho)=1-\rho$. 
\begin{theorem} [Finite-time blowup for the local ARZ model with linear velocity]\label{thm:arzblowupLWR}
Consider the local ARZ model \eqref{eq:ARZ} with linear velocity function $U$ satisfying \eqref{eq:Ulinear}. Suppose the initial data $(\rho_0, u_0)$ satisfy \eqref{A1}-\eqref{A4}, and there exists $x_0\in\R$ such that $u_0'(x_0)<0$. Then, the solution must develop a shock in finite time. More precisely, there exist a finite time $T_*>0$ and $x\in\R$ such that 
\[
\lim_{t\to T_*-}\partial_x \rho(t,x) = \infty\quad\text{and}\quad
\lim_{t\to T_*-}\partial_x u(t,x) = -\infty.
\]
\end{theorem}
Theorem \ref{thm:arzblowupLWR} shows that a finite-time blowup must happen as long as $u_0$ is not monotone increasing. 

We further explore a generalization of Theorem \ref{thm:arzblowupLWR} in the context of more general velocity functions, including \eqref{eq:Unonlinear}.
\begin{theorem} [Finite-time blowup for local ARZ model with general velocity]\label{thm:arzblowupGENERAL}
Consider the local ARZ model \eqref{eq:ARZ} with general velocity function $U$ satisfying \eqref{eq:Ucond}, and there exists $\beta\in(0,1]$ such that
\begin{equation}\label{eq:Ublowup}
    \frac{\rho U''(\rho)}{U'(\rho)}\geq-1, \quad \forall~\rho\in[0,\beta].
\end{equation}
Suppose the initial data $(\rho_0, u_0)$ satisfy \eqref{A1}-\eqref{A4}, and there exists $x_0\in\R$ such that $\rho_0(x_0)\leq\beta$ and $u_0'(x_0)<0$. Then, the solution must develop a shock in finite time. 
\end{theorem}

\begin{remark}
The condition \eqref{eq:Ublowup} on $U$ in Theorem \ref{thm:arzblowupGENERAL} is satisfied by the family of velocity functions $U(\rho) = (1-\rho)^J$ in \eqref{eq:Unonlinear}, with $\beta=\frac1J$. If we take $J=1$, the theorem reduces to Theorem \ref{thm:arzblowupLWR}.
\end{remark}

In contrast to the blowup observed in the local ARZ model, we demonstrate that the nonlocal look-ahead slowdown interaction in \eqref{eq:lookahead} can help preventing blowup. This phenomenon has been investigated in \cite{lee2022sharp,hamori2023sharp} for the first-order models \eqref{eq:basicn}, with a specific choice of the uniform weight function:
\begin{equation}\label{eq:kernel}
w(z) = 1,\quad \forall~z\geq0.
\end{equation}

We extend this phenomenon to the associated second-order nonlocal model \eqref{eq:main}. In addition to \eqref{A4}, we need to impose the following technical assumption on $\psi_0$:
\begin{enumerate}[label=(A5),ref=A5]
 \item \label{A5} $F_0(x):=\displaystyle\frac{\psi_0'(x)}{\rho_0(x)}$ and $G_0(x):=\displaystyle\frac{F_0'(x)}{\rho_0(x)}$ are bounded, for any $x\in\R$.
\end{enumerate}
The assumption \eqref{A5} requires $\psi_0'(x)=0$ if $\rho_0(x)=0$, namely there is no growth of $\psi_0$ at vacuum. 

We present our main theorem as follows.
\begin{theorem}[Global well-posedness]\label{thm:main}
Let $k>\frac32$. Consider the Cauchy problem \eqref{eq:main}-\eqref{eq:init} where the initial data $(\rho_0,u_0)$ satisfy the assumptions \eqref{A1}-\eqref{A5}. Assume $U$ satisfies \eqref{eq:Ucond} and \eqref{eq:Usmooth} (or \eqref{eq:Usmoothp}, with \eqref{A2} replaced by \eqref{A2p}). Assume $w$ satisfies \eqref{eq:kernel}.
Then, there exists a threshold function $\eta$, defined uniquely through \eqref{eq:etadyn}, such that if the initial data is subcritical, satisfying
\begin{equation}\label{eq:subcritical}
\rho_0'(x)\leq\eta(\rho_0(x))\quad\text{and}\quad \rho_0(x)\leq \rho_c, \quad \forall~x\in \R,	
\end{equation}
then the system admits a global smooth solution, namely $(\rho(t),u(t))$ satisfies \eqref{A1}-\eqref{A5} for any $t>0$.
\end{theorem}

\begin{remark}
 Figure \ref{fig:eta} illustrates the threshold function $\eta$. In particular, $\eta(0)=0$ and $\eta'(0)>0$. Therefore, the intersection of the subcritical region \eqref{eq:subcritical} and $\rho_0'(x)>0$ is non-empty. This means that there exists a class of non-trivial subcritical initial conditions where $\rho_0$ is not monotone decreasing.
 
 When $\psi_0\equiv0$, our system reduces to the first-order nonlocal model \eqref{eq:basicn}. The threshold function $\eta$ then coincides with $\sigma$, which is the threshold function for \eqref{eq:basicn} extensively studied in \cite{lee2022sharp,hamori2023sharp}.  It has been shown that $\sigma$ is bounded in $[0,\rhoc]$. 
While $\eta$ is bounded from above by $\sigma$, it is possible that $\eta$ becomes $-\infty$ at $\rho_*<\rhoc$. To make sense of \eqref{eq:subcritical}, we adopt the convention $\eta(\rho)=-\infty$ for $\rho\in[\rho_*,\rhoc]$, and the inequality $\rho_0'(x)\leq-\infty$ is always false. In any case, the subcritical region consists of points $(\rho_0(x),\rho'_0(x))$ in the phase plane lying beneath the graph of the threshold function $\eta$.
 
 We focus our discussion in the range where $\rho\leq\rhoc$, namely the flux $f$ is concave. When we allow the initial density to exceed $\rhoc$, there may arise another form of finite-time blowup where $\pa_x\rho(t,x)\to-\infty$. A comprehensive study on this type of blowup has been done in \cite{hamori2023sharp} for the first-order model \eqref{eq:basicn}, introducing another threshold function $\gamma$. Extending this analysis to the second-order model remains a subject for future investigation.
\end{remark}

\begin{figure}[htb]
\includegraphics[width=0.45\textwidth]{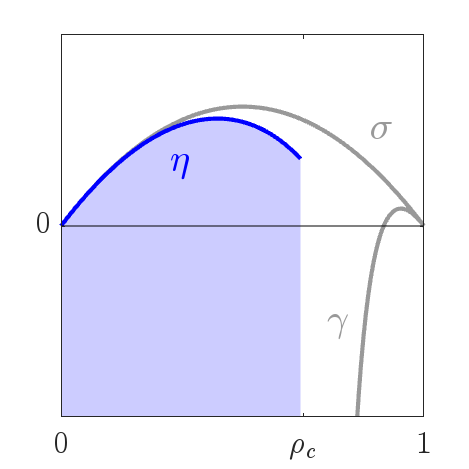}
\includegraphics[width=0.45\textwidth]{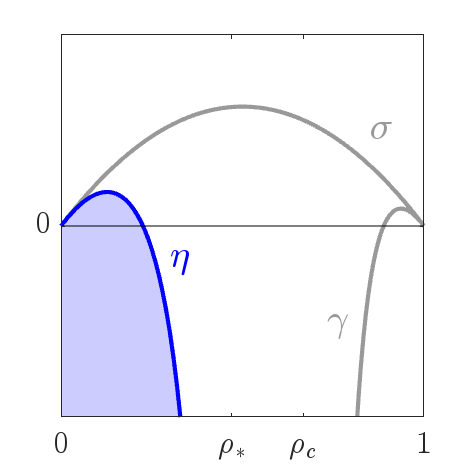}
\caption{Illustrations to the threshold function $\eta$ and the subcritical region in \eqref{eq:subcritical}. The flux is chosen as $f(\rho)=\rho(1-\rho)^2$. Depending on different choices of $\psi_0$, $\eta$ can either be bounded in $[0,\rhoc]$ (left), or it may blow up to $-\infty$ at $\rho_*<\rhoc$ (right).}\label{fig:eta}	
\end{figure}

The proof of Theorem \ref{thm:main} represents the most significant analytical challenge in this work. Beyond the nonlocality inherent in \eqref{eq:lookahead}, the dynamics of $\psi$ propagate with a different characteristic speed compared to $u$ (and $\rho$). The evolution of $\psi$ (as well as $\pa_x\psi$, $\pa_{xx}^2\psi$, etc.) along the characteristic path $X(t,x)$ relies on nonlocal information. To effectively handle the nonlocality, we introduce a comparison principle to compare the trajectories of $(\rho,\pa_x\rho)$ in the phase plane against the threshold function $\eta$, which is constructed solely from local information. We leverage a mediant inequality (Lemma \ref{lem:mediant}) to establish this comparison principle, a novel approach capitalizing on the unique structure of the phase plane dynamics.

\section{Local well-posedness and regularity criteria}\label{sec:lwp}

In this section, we develop a local well-posedness theory for the nonlocal ARZ model \eqref{eq:main}, proving Theorem \ref{thm:lwp}.

Suppose $(\rho,u)$ is a classical solution of \eqref{eq:main} with initial condition $(\rho_0,u_0)$ satisfying \eqref{A1}-\eqref{A4}. Our goal is to show that the assumptions \eqref{A1}-\eqref{A4} hold for $(\rho(t),u(t))$ at any time $t\in[0,T]$ as long as solution exists.

\subsection{A priori bounds}
We begin by establishing some a priori bounds that justify the propagation of assumptions \eqref{A1}, \eqref{A2} (or \eqref{A2p}), and \eqref{A4}.

By integrating the $\rho$-equation in \eqref{eq:main} with respect to $x$, we obtain
\[
\frac{d}{dt}\int_\R\rho(t,x)dx = -\int_\R \pa_x(\rho u) dx = 0.
\]
Therefore, the total mass is \emph{conserved} in time, namely:
\[
\int_\R \rho(t,x)dx = \int_\R \rho_0(x)dx = m,\quad \forall~t\geq0.
\]
Hence, the assumption \eqref{A1} holds in all time.

Next, we state the following maximum principles on $\rho$ and $u$, as well as the monotonicity of $\psi$, which contribute to the propagation of assumption \eqref{A2} and \eqref{A4} over all time. 

\begin{proposition}[Maximum principles]\label{prop:MP}
Let $(\rho,u)$ be a classical solution of \eqref{eq:NARZ} with initial data $(\rho_0,u_0)$ satisfying \eqref{A1}-\eqref{A4}.
Then, for any $t\in[0,T]$ and $x\in\R$, we have
\[
0\leq\rho(t,x)\leq 1, \quad 0\leq u(t,x)\leq 1,\quad \text{and}\quad \pa_x\psi(t,x)\geq0.
\]
\end{proposition}
\begin{proof}
We begin with the monotonicity of $\psi(t,\cdot)$. Differentiating the $\psi$-equation in \eqref{eq:main} with respect to $x$, we get
\[(\pa_t+u\pa_x)\pa_x\psi=-\pa_xu\,\pa_x\psi.\]
Integrating along the characteristic path $\widetilde{X}(t,x)$ where
\[ \pa_t\widetilde{X}(t,x)=u(t,\widetilde{X}(t,x)),\quad \widetilde{X}(0,x)=x,\]
we obtain that for any $x\in\R$,
\begin{equation}\label{eq:psix}
\pa_x\psi(t,\widetilde{X}(t,x))=\psi_0'(x)\exp\left(-\int_0^t \pa_xu(\tau,\widetilde{X}\big(\tau,x)\big)\,d\tau\right)\geq0,
\end{equation}
where we have used \eqref{A4}.
Since $(\rho,u)$ is a classical solution, $\widetilde{X}(t,\cdot): \R\to\R$ is a bijection. Therefore, $\pa_x\psi(t,x)\geq0$ for any $t\in[0,T]$ and $x\in\R$.

The same argument can be used to obtain $\rho(t,x)\geq0$ since $\rho$ and $\pa_x\psi$ satisfy the same equation, and $\rho_0(x)\geq0$ due to \eqref{A2}.

Next, we obtain the upper bound on $\rho$. Let us rewrite the $\rho$-equation in \eqref{eq:main} as the following dynamics along the characteristic path $X(t,x)$ defined in \eqref{eq:path}. 
\begin{equation}\label{eq:rhodynamics}
  \dot \rho = \pa_t\rho + \big(u+\rho U'(\rho) e^\mrhob\big)\pa_x\rho = -\rho\pa_x\psi + \rho U(\rho) e^\mrhob \pa_x\rhob.
\end{equation}
Observe that $\pa_x\psi\geq0$ and $U(1)=0$. Hence $\dot\rho\leq0$ when $\rho=1$. A standard comparison principle implies $\rho(t,X(t,x))\leq1$ as long as $\rho_0(x)\leq1$. 

Finally, we turn our attention to $u$. Compute the dynamics of $u$:
\begin{align*}
\pa_tu &= \pa_t\psi+\pa_t(U(\rho))e^\mrhob+U(\rho)\pa_t(e^\mrhob)\\
&=-u\big(\pa_x u -\pa_x(U(\rho)e^\mrhob)\big)-U'(\rho)\pa_x(\rho u)e^\mrhob+U(\rho)\pa_t(e^\mrhob)\\
&=-(u+\rho U'(\rho)e^\mrhob)\pa_xu+U(\rho)\big(\pa_t(e^\mrhob)+u\pa_x((e^\mrhob)\big).
\end{align*}
Then, the dynamics of $u$ along the characteristic path $X(t,x)$ reads:
\begin{align} 
\dot u &= \pa_tu+(u + \rho U'(\rho)e^\mrhob ) \pa_x u  = -U(\rho)e^\mrhob\big(\pa_t\rhob + u\pa_x\rhob\big)\nonumber\\
&= -U(\rho)e^\mrhob \int_0^\infty w(z)\Big( \pa_t \rho(x+z) + u(x)\pa_x\rho(x+z)\Big) dz\nonumber\\
&= -U(\rho)e^\mrhob \int_0^\infty w(z)\Big( -\pa_z (\rho(x+z)u(x+z)) + u(x)\pa_z\rho(x+z)\Big) dz\nonumber\\
&= -U(\rho)e^\mrhob \int_0^\infty w'(z) \rho(x+z)(u(x+z)- u(x)) dz.\label{eq:udynamics}
\end{align}
Here, $w'$ can be realized as weak derivative, if $w$ is not smooth. Note that the boundary term at $z=0$ in the integration by parts vanishes as $-w(0)\rho(x)u(x)+w(0)\rho(x)u(x)=0$.

To show $u(t,x)\leq1$, we argue by contradiction. Suppose the value of $u$ can exceed $1$, then there exist $t_*$ and $x_*$ such that the first breakthrough happens. Namely, 
\[u(t_*,x)\leq 1,\quad\forall~x\in\R,\quad\text{and}\quad u(t_*,x_*)=1.\]
From \eqref{eq:udynamics} and applying $U(\rho)e^\mrhob\geq0$, $w'(z)\leq0$, $\rho(x_*+z)\geq0$, and $u(x_*+z)-u(x_*)\leq0$, we obtain $\dot u(t_*,x_*)\leq0$. Hence, the breakthrough cannot happen. This leads to a contradiction.

The minimum principle $u(t,x)\geq0$ can be proved with a similar argument.
\end{proof}

\begin{remark}\label{rmk:psi}
 The propagation of the monotonicity assumption on $\psi$ \eqref{A4} plays a crucial role in the proof of the maximum principle for $\rho$. Without this assumption, we cannot deduce $\dot\rho\leq0$ from \eqref{eq:rhodynamics}, and the maximum principle does not follow. In fact, one can construct initial data satisfying \eqref{A1}-\eqref{A3} that lead to unrealistic solutions where $\rho(t,x)>1$ at some $(t,x)$.
 
An alternative method to circumvent the loss of the maximum principle, proposed in \cite{berthelin2008model}, considers a distinct type of pressure in the Aw-Rascle model \eqref{eq:AR} such that $p(\rho)$ becomes singular at $\rho=1$. However, this adjustment breaks the relation \eqref{eq:pUrelation}, causing the system to fall outside our framework. Instead, we introduce assumption \eqref{A4} to uphold the maximum principle.
\end{remark}

The following maximum principle on $\rho$ ensures the propagation of the assumption \eqref{A2p}, which is necessary if $U$ is not smooth at $\rho=1$, as discussed in Remark \ref{rmk:U}.

\begin{proposition}\label{prop:MP2}
Under the assumptions of Proposition \ref{prop:MP}, if we additionally assume the initial data $\rho_0$ satisfies \eqref{A2p}, then for any $t\in[0,T]$ and $x\in\R$, we have
\[
0\leq\rho(t,x)\leq \rhoM.
\]
\end{proposition}
\begin{proof}
 The proof differs slightly from the proof in Proposition \ref{prop:MP} since $U(\rhoM)\neq0$, so the last term in \eqref{eq:rhodynamics} does not vanish when $\rho=\rhoM$. Instead, we will show that the term is negative. 
 
 Suppose the value of $\rho$ can exceed $\rhoM$, then there exist $t_*$ and $x_*$ such chat the first breakthrough happens, namely
 \[\rho(t_*,x)\leq \rhoM,\quad \forall~x\in\R,\quad\text{and}\quad \rho(t_*,x_*)=\rhoM.\]
 Compute
 \begin{align}
    \pa_x\rhob(t_*,x_*) &= \int_0^\infty w(z)\pa_z\rho(t_*,x_*+z)dz =-w(0)\rho(x_*) - \int_0^\infty w'(z)\rho(t_*,x_*+z)dz\nonumber\\
    &\leq -w(0)\rhoM+\rhoM\int_0^\infty\big(-w'(z)\big)dz\leq -w(0)\rhoM + w(0)\rhoM = 0.\label{eq:rhobxdyn}
\end{align}
Then we apply \eqref{eq:rhodynamics} and obtain $\dot\rho(t_*,x_*)\leq0$. Hence the breakthrough cannot happen. This leads to a contradiction.
\end{proof}

The following bounds on the nonlocal interaction term will be used in our analysis.
From conservation of mass, maximum principle on $\rho$ and monotonicity of $w$, we obtain
\begin{equation} \label{eq:rhobbound}
 0\leq \rhob(t,x)=\int_0^\infty w(z)\rho(x+z)dz\leq w(0)\int_0^\infty\rho(x+z)dz\leq m w(0),
\end{equation}
which then implies
\begin{equation}
\label{eq:nonlocalbound}
e^{- mw(0)}\leq e^{-\rhob}\leq 1.
\end{equation}
Furthermore, a similar argument as \eqref{eq:rhobxdyn} yields
\begin{align*}
 |\pa_x\rhob(t,x)| &= \left|-w(0)\rho(x) + \int_0^\infty (-w'(z))\rho(t,x+z)dz\right|\\
 &\leq \max\left\{w(0)\rho(x),\int_0^\infty (-w'(z))\rho(t,x+z)dz \right\}
 \leq w(0).
\end{align*}
Consequently, we have
\begin{equation}\label{eq:nonlocalpax}
\|\pa_x(e^\mrhob)\|_{L^\infty} = \|e^\mrhob\|_{L^\infty}\|\pa_x\rhob\|_{L^\infty} \leq w(0).	
\end{equation}

\subsection{Energy estimates}
In this part, we aim to show the propagation of the regularity assumption \eqref{A3}. We will study the time evolution of $\|\pa_x\rho\|_{H^{k-1}}$ and $\|\pa_xu\|_{H^{k-1}}$.

It is more convenient for us to work with $(u,\psi)$ as these quantities are the \emph{Riemann invariants} for the local ARZ system \eqref{eq:ARZ}.

We denote by $E_k$ the homogeneous $\dot{H}^k$ energy on $u$ and $\psi$:
\[
E_k(t) = \|u(t,\cdot)\|_{\dot H^k}^2+\|\psi(t,\cdot)\|_{\dot H^k}^2.
\]

\begin{theorem}[Energy estimate]\label{thm:energy}
 Let $(\rho,u)$ be a classical solution of \eqref{eq:main} with initial data $(\rho_0,u_0)$ satisfying \eqref{A1}-\eqref{A4}. Then, for any $1\leq s\leq k$, we have
 \begin{equation}\label{eq:energy}
  E_s'(t) \leq C \Big(1 + \|\pa_x \rho\|_{L^\infty} +\|\pa_x u\|_{L^\infty} + \|\pa_x \psi\|_{L^\infty} \Big) \big(E_s + \|\rho\|_{L^2}^2\big),
\end{equation}
where the constant $C$ depends on
\[
C=C\left(k, m, w(0), \|U\|_{C^{\lceil k \rceil +1}([0,\rho_M])},\|U^{-1}\|_{C^{\lceil k\rceil}([U(\rho_M),1])}\right).
\]
\end{theorem}

The proof of Theorem \ref{thm:energy} relies on standard commutator estimates (Lemma \ref{lem:commutator}). However, the nonlinearity in $U$ and the nonlocal terms present additional challenges. We employ the composition estimate (Lemma \ref{lem:composition}) to address both the nonlinearity and nonlocality. Due to the length of the proof, we will defer it to the Appendix.

Now, we define the energy
\[Y(t)=E_1(t)+E_k(t)=\|\pa_x u(t,\cdot)\|_{L^2}^2+\|\pa_x u(t,\cdot)\|_{L^2}^2+\|\pa_x \psi(t,\cdot)\|_{\dot H^{k-1}}^2+\|\pa_x \psi(t,\cdot)\|_{\dot H^{k-1}}^2.\]
Clearly, $Y(t)$ is equivalent to $\|\pa_x u(t,\cdot)\|_{ H^{k-1}}^2+\|\pa_x \psi(t,\cdot)\|_{ H^{k-1}}^2$. 
Under assumption \eqref{A3}, we know $Y(0)$ is finite.

The energy estimate \eqref{eq:energy} provides a control of the growth of $Y$:
\begin{align}\label{eq:ygrowth}
    Y'(t) \leq C \Big(1 + \|\pa_x \rho\|_{L^\infty} +\|\pa_x u\|_{L^\infty} + \|\pa_x \psi\|_{L^\infty} \Big)(Y(t)+\|\rho\|_{L^2}^2).
\end{align}
Note that $\|\rho\|_{L^2}$ has an a priori bound thanks to \eqref{A1} and \eqref{A2}. Indeed, we have
\begin{equation}\label{eq:rhoL2}
 \|\rho\|_{L^2}^2\leq\|\rho\|_{L^1}\|\rho\|_{L^\infty}\leq m.	
\end{equation}

Next, we show $Y(t)$ is bounded locally in time. For $k>3/2$, we apply Sobolev embedding and \eqref{eq:rhoHk2} to obtain  
\begin{align*}
&\|\pa_x \rho\|_{L^\infty} +\|\pa_x u\|_{L^\infty} + \|\pa_x \psi\|_{L^\infty} \leq C\big(\|\pa_x \rho\|_{H^{k-1}} +\|\pa_x u\|_{H^{k-1}} + \|\pa_x \psi\|_{H^{k-1}}\big)\\
&\leq C\big(\|\pa_x u\|_{H^{k-1}} + \|\pa_x \psi\|_{H^{k-1}}+\|\rho\|_{L^2} \big)\leq C(Y(t)+m)^{1/2}.
\end{align*}
Then the equation \eqref{eq:ygrowth} has the form
\[Y'(t)\leq C(Y(t)+1)^{3/2}.\]
Classical Cauchy-Lipschitz theory implies local well-posedness: there exists a time $T>0$, depending on $Y(0)$ and $C$, such that $Y(t)$ is bounded for any $t\in[0,T]$.

As long as $Y(t)$ is bounded, we have $\pa_x u(t,\cdot)\in H^{k-1}$. Also, $\pa_x \rho(t,\cdot)\in H^{k-1}$ due to \eqref{eq:rhoHk2} and \eqref{eq:rhoL2}.
Hence, the regularity assumption \eqref{A3} holds for $(\rho(t),u(t))$.

\subsection{Regularity criteria}
In this part, we show that the regularity criterion \eqref{eq:condition} implies the boundedness of $Y(t)$, thereby ensuring the smoothness of the solution $(\rho(t),u(t))$.

By applying the Gr\"onwall inequality to \eqref{eq:ygrowth}, we see that
\[
    Y(T)\leq (Y(0)+m)\exp \left(C\int_0^T \big(1 + \|\pa_x \rho\|_{L^\infty} +\|\pa_x u\|_{L^\infty} + \|\pa_x \psi\|_{L^\infty} \big) dt\right).
\]
Therefore, a sufficient condition to ensure boundedness of $Y(T)$ is
\begin{equation}\label{eq:condition2}
\int_0^T \big(\|\pa_x \rho\|_{L^\infty} +\|\pa_x u\|_{L^\infty} + \|\pa_x \psi\|_{L^\infty} \big) dt<\infty.	
\end{equation}

To obtain \eqref{eq:condition}, we show that $\|\pa_x u\|_{L^\infty}$ and $\|\pa_x \psi\|_{L^\infty}$ are controlled by $\|\pa_x \rho\|_{L^\infty}$.
\begin{proposition}\label{prop:uxpsix} For any $t\in[0,T]$, we have
\begin{align*}
  \|\pa_x u(t,\cdot)\|_{L^\infty} &\leq C\left[\exp\left(C\int_0^t \|\pa_x\rho(\tau,\cdot)\|_{L^\infty}d\tau\right)+\|\pa_x\rho(t,\cdot)\|_{L^\infty} + 1 \right],\\
  \|\pa_x \psi(t,\cdot)\|_{L^\infty} &\leq C\exp\left(C\int_0^t \|\pa_x\rho(\tau,\cdot)\|_{L^\infty}d\tau\right),
\end{align*}
where the constant $C$ depends on $\|\psi_0'\|_{L^\infty}$, $\|U\|_{C^1([0,\rhoM])}$ and $w(0)$.
\end{proposition}
\begin{remark}
	The quantity $\|\psi_0'\|_{L^\infty}$ is finite due to assumption \eqref{A3} and Sobolev embedding.
\end{remark}
\begin{proof}[Proof of Proposition \ref{prop:uxpsix}]
First, the lower bound $\pa_x\psi\geq0$ follows from \eqref{A4}. 

To obtain a lower bound on $\pa_xu$, we differentiate the $u$-equation in \eqref{eq:main} in $x$ and get
\begin{equation}\label{eq:uxlower}
\pa_xu = \pa_x\psi+U'(\rho)e^\mrhob\pa_x\rho- U(\rho)\pa_x(e^\mrhob)\geq -\|U\|_{C^1([0,\rhoM])}|\pa_x\rho|-w(0).
\end{equation}
where we have used \eqref{A4}, \eqref{eq:nonlocalbound} and \eqref{eq:nonlocalpax}.

Next, we derive an upper bound for $\pa_x\psi$. 
From \eqref{eq:psix} and applying \eqref{eq:uxlower}, we get
\[
 \pa_x \psi(t,\widetilde{X}(t,x))\leq\psi'_0(x)\exp\left(\|U\|_{C^1([0,\rhoM])}\int_0^t\|\pa_x\rho(\tau,\cdot)\|_{L^\infty}\,d\tau+w(0)t\right).
\]

Finally, we work on the upper bound of $\pa_xu$.
\begin{align*}
\pa_xu &= \pa_x\psi+U'(\rho)e^\mrhob\pa_x\rho- U(\rho)\pa_x(e^\mrhob)\\
&\leq \psi'_0\exp\left(\|U\|_{C^1([0,\rhoM])}\int_0^t\|\pa_x\rho(\tau,\cdot)\|\,d\tau+w_0t\right)+\|U\|_{C^1([0,\rhoM])}|\pa_x\rho|+w(0).
\end{align*}
\end{proof}

Using Proposition \ref{prop:uxpsix}, we argue that the regularity criterion \eqref{eq:condition} implies \eqref{eq:condition2} and thus ensures the regularity of the solution. This finishes the proof of Theorem \ref{thm:lwp}.

\subsection{Finite time blowup for local ARZ model}
The regularity criterion \eqref{eq:condition} suggests that the solution to our system \eqref{eq:main} is globally regular as long as $\pa_x\rho$ is bounded. However, achieving boundedness of $\pa_x\rho$ is challenging. Here, we illustrate  that the criterion \eqref{eq:condition} is violated in finite time for generic initial data for the local ARZ model \eqref{eq:ARZ}, as stated in Theorems \ref{thm:arzblowupLWR} and \ref{thm:arzblowupGENERAL}.

We provide a proof of Theorem \ref{thm:arzblowupGENERAL} below, and Theorem \ref{thm:arzblowupLWR} follows as a special case of Theorem \ref{thm:arzblowupGENERAL} where \eqref{eq:Ublowup} is satisfied with $\beta=1$.

\begin{proof}[Proof of Theorem \ref{thm:arzblowupGENERAL}]
Denote $v:=\pa_xu$. We begin by differentiating the $u$-equation in \eqref{eq:Z} with respect to $x$, and obtaining the dynamics of $v$ along the characteristic path $X(t,x)$:
\[
\dot v = \big(\pa_t + (u+\rho U'(\rho)\pa_x\big) v = -v\big(v+\pa_x(\rho U'(\rho))\big) = -v\big(v+(\rho U''(\rho)+U'(\rho))\pa_x\rho\big).
\]
Recall that from the $u$-equation in \eqref{eq:ARZ}, we have the relation 
\begin{equation}\label{eq:vrelationARZ}
v = \pa_x\psi+U'(\rho)\pa_x \rho. 	
\end{equation}
Then
\begin{align*}
v+(\rho U''(\rho)+U'(\rho))\pa_x\rho&=v+\frac{\rho U''(\rho)+U'(\rho)}{U'(\rho)}(v-\pa_x\psi)\\
&=\left(2+\frac{\rho U''(\rho)}{U'(\rho)}\right)v-\left(1+\frac{\rho U''(\rho)}{U'(\rho)}\right)\pa_x\psi.
\end{align*}

Now, take $x_0\in\R$ such that $\rho_0(x_0)\leq\beta$ and $u_0'(x_0)<0$. From the $\rho$-equation in \eqref{eq:ARZ} and \eqref{A4} we get
\[
\dot \rho = \pa_t \rho+(u+\rho U'(\rho))\pa_x \rho = -\rho\pa_x\psi\leq 0.
\]
This implies
\[
\rho(t,X(t,x_0))\leq \rho_0(x_0)\leq\beta,
\]
for any time $t\geq0$, as long as classical solution exists.

Then, the assumption \eqref{eq:Ublowup} ensures 
\begin{equation}\label{eq:Ublowup2}
 \frac{\rho(t,X(t,x_0)) U''(\rho(t,X(t,x_0)))}{U'(\rho(t,X(t,x_0)))}\geq-1.	
\end{equation}

Next, we estimate the dynamics of $v$ along $X(t,x_0)$ using \eqref{A4} and \eqref{eq:Ublowup2}, and obtain
\[\dot{v}(t,X(t,x_0))\leq -v(t,X(t,x_0))(1\cdot v(t,X(t,x_0))-0)=-v^2(t,X(t,x_0)),\]
with initial data $v_0(x_0)<0$. Clearly, there exists a finite time $T$ such that 
\begin{equation}\label{eq:vblowup}
 \lim_{t\to T-} v(t,X(t,x_0))=-\infty.
\end{equation}
unless blowup happens earlier.

Owing to \eqref{eq:vrelationARZ} and $U'(\rho)<0$, we have
$\pa_x\rho=\frac{v -\pa_x\psi}{U'(\rho)}\geq \frac{v}{U'(\rho)}.$
Hence, \eqref{eq:vblowup} implies
\[ \lim_{t\to T-} \rho(t,X(t,x_0))=\infty.\]
\end{proof}

\section{Global well-posedness}\label{sec:gwp}

We now turn our attention to global well-posedness for our nonlocal ARZ model \eqref{eq:main}. The goal is to show that, with the help of the nonlocal interaction, we can identify a class of subcritical initial data  where there is no breach of the regularity criterion \eqref{eq:condition} for any finite time $T$. Consequently, the solution remains globally regular. This feature distinguishes our model from the local ARZ model \eqref{eq:ARZ}, where generic initial data result in finite-time shock formation.

We will work with the nonlocal interaction \eqref{eq:lookahead} and the special weight function \eqref{eq:kernel}. The structure of this weight function gives rise to the relation:
\begin{equation}\label{eq:rhobx}
 \pa_x \rhob = -\rho,
\end{equation}
which will play a crucial role in our analysis.

Throughout the section, we assume that the flux $f$ has the form
\[f(\rho)=\rho U(\rho),\]
where $U$ satisfies \eqref{eq:Ucond}. We also assume $f$ satisfies \eqref{eq:fluxcc}, namely $f$ is either concave $(\rhoc=1)$ or concave-convex $(\rhoc<1)$. In particular, \eqref{eq:Ucond} and \eqref{eq:fluxcc} together imply
\begin{equation}\label{eq:f0}
f(0)=0,\quad f'(0)>0,\quad f''(0)<0,\quad\text{and}\quad f(1)=0.	
\end{equation}

Let us characterize our approach as follows:
\begin{enumerate}
    \item First, establish the dynamics of $(\rho,\pa_x \rho)$ along characteristic paths, and describe the corresponding trajectories in the phase plane.
    \item Construct a threshold function $\eta$ in the phase plane, based on the threshold function $\sigma$ that has been used to describe the region of subcritical initial data for the first-order nonlocal model \eqref{eq:basicn}, as in \cite{lee2022sharp,hamori2023sharp}.
    \item Establish a comparison principle to connect $\eta$ to the trajectories of the dynamics of $(\rho,\pa_x \rho)$, and show that $\eta$ determines a region of subcritical initial data for our model \eqref{eq:main}. This verifies Theorem \ref{thm:main}.
\end{enumerate}

\subsection{Phase plane dynamics on $(\rho, \pa_x\rho)$}

 We proceed with a derivation of a coupled dynamics on $(\rho,\pa_x \rho)$ along the characteristic path $X(t,x)$. Let us denote
$$d\coloneqq \pa_x \rho,\quad f(\rho) \coloneqq \rho U(\rho), \quad F \coloneqq \pa_x\psi/\rho, \quad G \coloneqq \pa_xF/\rho. $$

Note that the auxiliary quantities $F$ and $G$ satisfy the transport equations:
\[\pa_tF+u \pa_xF = 0, \quad \pa_tG+u \pa_xG = 0.\]
This leads to the a priori bounds:
\begin{equation}\label{eq:FGbound}
\|F(t,\cdot)\|_{L^\infty} = \|F_0\|_{L^\infty},\quad \|G(t,\cdot)\|_{L^\infty} = \|G_0\|_{L^\infty},
\end{equation}
justifying that the assumption \eqref{A5} propagates in time.
Moreover, \eqref{A4} implies that 
\begin{equation}\label{eq:Fpositive}
F(t,x)\geq0,\quad \forall~x\in\R,\,\,t\in[0,T].	
\end{equation}
On the other hand, we do not impose any sign condition on $G_0$. 

Now, we derive the dynamics of $(\rho,d)$ along the characteristic path $X(t,x)$. From the dynamics of $\rho$, we apply \eqref{eq:rhobx} to \eqref{eq:rhodynamics} and get:
\[
  \dot \rho = -\rho\pa_x\psi + \rho U(\rho) e^\mrhob (-\rho)= -e^\mrhob\rho f(\rho) -\rho^2F.
\]
Differentiating the $\rho$-equation in \eqref{eq:main} with respect to $x$, we acquire:
\begin{align*}
 \dot{d} &= \pa_t d+\big(u+\rho U'(\rho) e^\mrhob\big) \pa_xd =  \pa_t d+\big(f'(\rho)e^\mrhob + \psi\big) \pa_xd\\
  &=-e^\mrhob \Big(f''(\rho) d^2 + (2\rho f'(\rho) + f(\rho)) d   + f(\rho) \rho^2\Big)  \nonumber -3\pa_x\psi d- \rho\pa^2_{xx}\psi, \nonumber\\
  &=  -e^\mrhob \Big(f''(\rho) d^2 + (2\rho f'(\rho) + f(\rho)) d   + f(\rho) \rho^2\Big)   -3\rho d F- \rho^3 G. 
\end{align*}
All together, we get the coupled dynamics
\begin{equation}\label{eq:coupled}
 \begin{cases}
    \dot \rho = -e^\mrhob\rho f(\rho) -\rho^2F\\    
   \dot d = -e^\mrhob \Big(f''(\rho) d^2 + [2\rho f'(\rho)   + f(\rho)] d   + \rho^2 f(\rho) \Big) -3\rho d F- \rho^3 G.
 \end{cases}
\end{equation}
Trajectories in the $(\rho,d)$ phase plane which are governed by \eqref{eq:coupled} may be expressed as $d=\d(\rho)$, in which case they must satisfy the ordinary differential equation:
    \begin{equation}\label{eq:neode}
  \d'(\rho)=\frac{e^\mrhob \Big(f''(\rho) \d^2 + [2\rho f'(\rho)   + f(\rho)] \d   + \rho^2 f(\rho) \Big) +3\rho \d F+ \rho^3 G }{e^\mrhob \rho f(\rho) +\rho^2F}.
   \end{equation}
    Our goal will be to analyze trajectories in the phase plane and determine whether they are bounded or not. We will construct a subcritical region
\[\Sigma \subset [0,1]\times\R,\]
such that for any initial data $(\rho_0,d_0)\in\Sigma$, the solution of \eqref{eq:neode} with $\d(\rho_0)=d_0$ is bounded. The boundedness of all trajectories $\d$ can ensure the regularity criterion \eqref{eq:condition}, thereby implying global well-posedness of \eqref{eq:main}.
    
Understanding the behavior of $\eqref{eq:neode}$ is a challenging task, particularly because the equation is determined by the nonlocal quantity $e^\mrhob$ which depends globally on $\rho$. Our strategy is to compare the system with the first-order nonlocal model \eqref{eq:basicn}-\eqref{eq:lookahead}, which is a special case of our second-order model \eqref{eq:main} with $\psi\equiv0$, and hence $F, G\equiv0$. 

\subsection{The threshold $\sigma$ for the first-order model}
The behavior of \eqref{eq:neode} can be thoroughly studied in \cite{lee2022sharp,hamori2023sharp} for the first-order model, where $F,G\equiv0$. In this case, \eqref{eq:neode} reduces to
\begin{equation}\label{eq:eode}
  \d'(\rho)=\frac{ f''(\rho) \d^2 + (2\rho f'(\rho) + f(\rho)) \d   + \rho^2 f(\rho)  }{ \rho f(\rho) }.
\end{equation}
A main feature of \eqref{eq:eode} is that it does not depend on any nonlocal information.

There is a critical threshold $\sigma(\rho)$ so that trajectories originating beneath this curve are bounded from above. We restate the result presently:
    
\begin{theorem}\label{thm:sigma}
Let $F=G=0$. There exists a unique trajectory $\sigma$ such that for any 
\begin{equation}\label{eq:sigma0}
 (\rho_0,d_0)\in  \big\{(\rho,d) : 0\leq\rho\leq\rhoc, d\leq\sigma(\rho)\big\},	
\end{equation}
the solution of the coupled dynamics \eqref{eq:coupled} with initial condition $(\rho_0,d_0)$ is bounded in all time.

Moreover, the threshold function $\sigma$ can be uniquely defined by:
\begin{equation}\label{eq:sigmaode}
\sigma'(\rho) = \frac{f''(\rho)\sigma^2+(2\rho f'(\rho) + f(\rho))\sigma+\rho^2f(\rho)}{\rho f(\rho)}=:\mathcal{F}(\rho,\sigma),
\end{equation}
with initial conditions
\[
  \sigma(0) = 0,\quad\text{and}\quad  \sigma'(0) = -\frac{2f'(0)}{f''(0)}>0.
\]
In particular, when $f(\rho) = \rho(1-\rho)^J$ with $J\geq1$, the threshold function $\sigma$ is explicitly given as 
\[
    \sigma(\rho) = \frac{\rho(1-\rho)}{J}.
\]
\end{theorem}

For the proof of the theorem, we refer to \cite{hamori2023sharp}. Here we make some remarks about the theorem that are relevant to our discussion.

The threshold function $\sigma$ is a trajectory that satisfies \eqref{eq:eode} with $\sigma(0)=0$. However, there are infinite many trajectories of this type. It has been proven that $\sigma$ is the only one with $\sigma'(0)\neq0$. Since trajectories cannot cross each other when $\rho\in(0,1)$, all trajectories initiated from the subcritical region in \eqref{eq:sigma0} are bounded from above by $\sigma$. 

The threshold condition $d\leq\sigma(\rho)$ is \emph{sharp}. Indeed, it has been shown in \cite{hamori2023sharp} that for any initial condition $(\rho_0,d_0)$ such that $0\leq\rho\leq\rhoc$ and $d_0>\sigma(\rho_0)$, the coupled dynamics \eqref{eq:coupled} must develop a finite-time singularity where $d(t)\to\infty$.
Sharp results are also available for $\rhoc<\rho\leq1$ when the flux is concave-convex. Another threshold function $\gamma$ is introduced to distinguish another type of finite-time singularity formation, where $d(t)\to-\infty$. See Figure \ref{fig:eta} for an illustration of $\sigma$ and $\eta$. We omit the details, and will focus on the region when $\rho\leq\rhoc$.

\subsection{A new threshold $\eta$}
We now construct a threshold function that helps us to determine a subcritical region for our second-order model \eqref{eq:main}. The idea is to compare the “worst case" behavior of \eqref{eq:neode} with that of \eqref{eq:eode}.

To effectively control the contribution of the $F$ and $G$ by the local dynamics \eqref{eq:main}, we investigate the structure of the right-hand side of \eqref{eq:neode}, and make use of the following mediant inequality. 
\begin{lemma}[Mediant inequality]\label{lem:mediant}
 Let $A_0,A_1\in \R$ and $B_0, B_1>0$. Then
 \[\frac{A_0+A_1}{B_0+B_1}\geq\min\left\{\frac{A_0}{B_0},\frac{A_1}{B_1}\right\}.\]
\end{lemma}
\begin{proof}
 Observe that 
 \[
  \frac{A_0+A_1}{B_0+B_1} = \frac{A_0}{B_0} \cdot \frac{B_0}{B_0+B_1} + \frac{A_1}{B_1}  \cdot \frac{B_1}{B_0+B_1} = \theta  \cdot \frac{A_0}{B_0}+(1-\theta) \cdot \frac{A_1}{B_1},
 \] 
 where $\theta \in [0,1]$. The mediant $\frac{A_0+A_1}{B_0+B_1}$ therefore lies in between between $\frac{A_0}{B_0}$ and $\frac{A_1}{B_1}$.
   \end{proof}

We apply Lemma \ref{lem:mediant} to \eqref{eq:neode} with 
\begin{align*}
&A_0 = e^\mrhob \Big(f''(\rho) \d^2 + (2\rho f'(\rho) + f(\rho)) \d   + \rho^2 f(\rho) \Big) + \rho^3 G,\quad A_1 = 3\rho \d F,\\
&B_0 = e^\mrhob \rho f(\rho),\quad\text{and}\quad B_1=\rho^2F,
\end{align*}
and obtain the following bound on the trajectory
\begin{align}\label{eq:comparisondynamics}
  \d'(\rho) &\geq \min\left\{\frac{f''(\rho) \d^2 + (2\rho f'(\rho) + f(\rho)) \d +   \rho^2 f(\rho) +\frac{\rho^3G}{e^\mrhob}}{\rho f(\rho)},\frac{3\rho \d F}{\rho^2F}\right\}\nonumber\\
  &\geq \min\left\{\frac{f''(\rho) \d^2 + (2\rho f'(\rho) + f(\rho)) \d +   \rho^2 f(\rho) -C_\eta\rho^3}{\rho f(\rho)},\frac{3 \d}{\rho}\right\},
   \end{align}
where 
\[C_\eta \coloneqq \|G_0\|_{L^\infty}e^{mw(0)}.\]
Note that the bound \eqref{eq:Fpositive} on the non-negativity of $F$ is crucially used in order to apply the mediant inequality. When $F=0$, the inequality holds trivially.

Importantly, the right-hand side of \eqref{eq:comparisondynamics} depends solely on local information. Therefore, we will construct a new threshold function $\eta$ satisfying \eqref{eq:comparisondynamics} with equality:
\begin{equation}\label{eq:etaode}
 \eta'(\rho)=\min\left\{\frac{f''(\rho) \eta^2 + (2\rho f'(\rho) + f(\rho)) \eta +   \rho^2 f(\rho) -C_\eta\rho^3}{\rho f(\rho)},\frac{3 \eta}{\rho}\right\}=:\mathcal{G}(\rho,\eta),
\end{equation}
and with $\eta(0)=0$. The function $\eta$ will serve as an upper bound of $\d$ through a \emph{comparison principle}, see Lemma \ref{lem:comparison}.

Similar to $\sigma$, the threshold function $\eta$ satisfying \eqref{eq:etaode} and $\eta(0)=0$ is not uniquely defined, as $\mathcal{G}(0,0)$ is not well-defined. We argue that the behaviors of $\mathcal{F}$ and $\mathcal{G}$ are similar around $(0,0)$, allowing us to construct a unique function $\eta$ following a similar approach to the construction of $\sigma$ in \cite{hamori2023sharp}.
\begin{proposition}
 Let $\eps>0$ and $y\in C^1{([0,\eps])}$ with $y(0)=0$. Then,
 \begin{equation}\label{eq:FGequal}
\lim_{\rho\to0^+}\mathcal{G}(\rho,y(\rho)) = \lim_{\rho\to0^+}\mathcal{F}(\rho,y(\rho)). 	
 \end{equation}
\end{proposition}\label{prop:FGequal}
\begin{proof}
From the definitions of $\mathcal{F}$ and $\mathcal{G}$ in \eqref{eq:sigmaode} and \eqref{eq:etaode}, we have the relation:	
\begin{equation}\label{eq:FGrelation}
\mathcal{G}(\rho,y) = \min\left\{\mathcal{F}(\rho,y)-\frac{C_\eta\rho^2}{f(\rho)},\frac{3y}{\rho}\right\}.	
\end{equation}
In light of \eqref{eq:f0}, we compute the limits:
\[
\lim_{\rho\to0^+} \frac{C_\eta\rho^2}{f(\rho)} = \lim_{\rho\to0^+} \frac{2C_\eta\rho}{f'(\rho)} = 0,\quad \lim_{\rho\to0^+}\frac{3y(\rho)}{\rho}=3y'(0).
\]
Therefore, we have
\[\lim_{\rho\to0^+}\mathcal{G}(\rho,y(\rho)) = \min\left\{\lim_{\rho\to0^+}\mathcal{F}(\rho,y(\rho)),3y'(0)\right\}.\]
Furthermore, we observe
\begin{align*}
 \lim_{\rho\to0^+}\mathcal{F}(\rho,y(\rho))&=\lim_{\rho\to 0_+}\frac{f''(\rho) y(\rho)^2 + (2\rho f'(\rho) + f(\rho)) y(\rho) + \rho^2 f(\rho)}{\rho f(\rho)}\\
 &= \frac{f''(0)}{f'(0)} y'(0)^2 + 3y'(0)\leq 3y'(0).
\end{align*}
This directly implies \eqref{eq:FGequal}, thanks to \eqref{eq:f0}.
\end{proof}

We apply Proposition \ref{prop:FGequal} to \eqref{eq:etaode} and obtain
\[\eta'(0)=\lim_{\rho\to0^+}\mathcal{G}(\rho,\eta(\rho))=\frac{f''(0)}{f'(0)} \eta'(0)^2 + 3\eta'(0).\]
Thus we have either $\eta'(0)=0$ or $\eta'(0)=\sigma'(0)=\frac{-2f'(0)}{f''(0)}>0$. 

Among all the candidates, we shall select the function $\eta$ such that $\eta'(0)=\sigma'(0)>0$. Moreover, we claim that such a choice of $\eta$ is unique.

\begin{proposition} \label{prop:eta}
There exists a unique trajectory $\eta$ that satisfies
\begin{equation}\label{eq:etadyn}
  \begin{cases}
  	\eta'(\rho) = \mathcal{G}(\rho,\eta(\rho)),\\
  	\eta(0) = 0,\quad  \eta'(0) = \sigma'(0)>0.
  \end{cases}
\end{equation}
\end{proposition}
The proof of Proposition \ref{prop:eta} closely parallels the proof of the well-posedness of $\sigma$ as stated in Theorem \ref{thm:sigma}. A detailed proof can be found in \cite[Theorem 3.1]{hamori2023sharp}. Extending the proof to $\eta$ is straightforward, utilizing Proposition \ref{prop:FGequal} to handle the local existence part. We omit the details of the proof for simplicity.

We now offer comments on the shape of the threshold function $\eta$. First of all, since $\eta(0)=0$ and $\eta'(0)>0$, we know that $\eta(\rho)$ is positive for small $\rho$, namely there exist a number $\rho_\eta>0$ such that
\[\eta(\rho)>0, \quad \rho\in (0,\rho_\eta).\]

Next, we observe from \eqref{eq:FGrelation} that $\mathcal{G}(\rho,y)\leq\mathcal{F}(\rho,y)$. This implies $\eta'(\rho)\leq\sigma'(\rho)$, and consequently
\[\eta(\rho)\leq\sigma(\rho),\quad\forall~\rho\in[0,\rhoc].\]
Since $\sigma$ is bounded on $[0,\rhoc]$, so is $\eta$, namely there exists a constant $M$ such that
\begin{equation}\label{eq:etaM}
\eta(\rho)\leq M<\infty,\quad\forall~\rho\in[0,\rhoc].	
\end{equation}

Finally, it's worth noting that while it's established in \cite[Proposition 3.2]{hamori2023sharp} that  $\sigma\geq0$, the same doesn't hold true for $\eta$. In fact, not only can $\eta$ become negative,  it may also blow up to $-\infty$. Namely, either $\eta$ is lower bounded, or there exists a number $\rho_*\in(0,\rhoc]$ such that 
\[\lim_{\rho\to \rho_*^-}\eta(\rho) = -\infty.\]

%
\subsection{Global behavior of solutions}
With the threshold function $\eta$ in hand, we are ready to explore the global behavior of the coupled dynamics \eqref{eq:coupled}. We state the following global well-posedness theorem for subcritical initial data.
\begin{theorem}\label{thm:subcritical}
  Consider the coupled system \eqref{eq:coupled} with subcritical initial data 
  \begin{equation}\label{eq:subinit}
    (\rho_0,d_0)\in\Sigma:=\big\{(\rho,d):0\leq\rho\leq\rho_c,\,d\leq\eta(\rho)\big\}. 	
  \end{equation}
Then the solution $(\rho,d)$ exists in all time and
       $$\lim_{t\to \infty} \rho(t) = 0, \quad \lim_{t \to \infty} d(t) = 0. $$
\end{theorem} 

Theorem \ref{thm:subcritical} is a non-trivial extension of Theorem \ref{thm:sigma}. The main challenges lie in dealing with the dynamics described by \eqref{eq:neode}, which involve nonlocal quantities such as $F$, $G$, and $e^\mrhob$. These quantities evolve over time, and we only have certain a priori bounds on them. To address this challenge, we establish a comparison principle that links the trajectories governed by \eqref{eq:neode} with the threshold function $\eta$.\begin{lemma}[Comparison principle]\label{lem:comparison}
 Let $\d$ and $\eta$ be two functions defined in $[0,\rho_0]$ that satisfy \eqref{eq:comparisondynamics} and \eqref{eq:etaode}, respectively. Suppose 
 \[\d(\rho_0)\leq\eta(\rho_0).\]
 Then, we have
 \begin{equation}\label{eq:comparison}
  \d(\rho)\leq\eta(\rho),\quad \forall~\rho\in[0,\rho_0]. 	
 \end{equation}
\end{lemma}
\begin{proof}
We argue by contradiction. Suppose \eqref{eq:comparison} is false. Then there must exist a $\rho_*\in(0,\rho_0]$ and a small $\eps\in(0,\rho_*)$ such that
\[\d(\rho_*)=\eta(\rho_*),\quad\text{and}\quad \d(\rho)>\eta(\rho)\quad\forall~\rho\in[\rho_*-\eps,\rho_*).\]
Note that $\rho_*$ is the largest $\rho$ when the comparison starts to fail.

On the other hand, leveraging the mediant inequality, from \eqref{eq:comparisondynamics} and \eqref{eq:etaode}, we know $d'(\rho_*)\geq\eta'(\rho_*)$. It is easy to check  $\mathcal{G}$ is uniformly Lipschitz in $\eta$ for $\rho\in [\rho_*-\eps,\rho_*]$, we have
\[\d'(\rho)-\eta'(\rho)\geq \mathcal{G}(\rho,\d(\rho)) - \mathcal{G}(\rho,\eta(\rho))\geq-\mathcal{G}_{Lip} (\d(\rho)-\eta(\rho)), \quad\forall~\rho\in[\rho_*-\eps,\rho_*],\]
where we denote $\mathcal{G}_{Lip}$ the Lipschitz constant. Integration along $[\rho_*-\eps,\rho_*]$ yields
\[\d(\rho_*-\eps)-\eta(\rho_*-\eps)\leq (\d(\rho_*)-\eta(\rho_*))e^{\mathcal{G}_{Lip}
\eps}=0.\]
This leads to a contradiction. 
\end{proof}

Now we are ready to prove Theorem \ref{thm:subcritical}.   
\begin{proof}[Proof of Theorem \ref{thm:subcritical}]
First, we establish the bounds on $\rho$. From the $\rho$-equation in \eqref{eq:coupled} and the bound \eqref{eq:Fpositive}, we get $\dot\rho\leq 0$ and hence $\rho(t)\leq\rho_0$. On the other hand, $\rho=0$ is a stationary state. Therefore, we have the bounds
\begin{equation}\label{eq:rhobound}
0\leq\rho(t)\leq\rho_0,\quad\forall~t\geq0.	
\end{equation}
Furthermore, we have the decay estimate:
\begin{equation}\label{eq:rhodecay}
\dot\rho\leq -e^\mrhob\rho^2U(\rho)-0 \leq -e^{-mw(0)}U(\rho_0)\rho^2.	
\end{equation}
Thus we deduce that $\rho(t)\to0$ as $t\to\infty$.
   
Next, we establish the bounds on $d$.
The upper bound can be obtained using the comparison principle \eqref{eq:comparison} and the bound in \eqref{eq:etaM}:
\[d(t) = \d(\rho(t))\leq \eta(\rho(t))\leq M,\quad \forall~t\geq0.\]
Note that we have used \eqref{eq:rhobound} in order to apply the comparison principle.

For the lower bound on $d$, we proceed with a similar argument as \cite[Theorem 4.1(c)]{hamori2023sharp}, and use a priori bounds \eqref{eq:FGbound}, \eqref{eq:Fpositive} and \eqref{eq:nonlocalbound} on the nonlocal terms. From the $d$-equation in \eqref{eq:coupled}, when $\rho<\rhoc$ and $d\leq0$ we have
\begin{align}\label{eq:dlowerbound}
   \dot d &= -e^\mrhob \Big(f''(\rho) d^2 + (2\rho f'(\rho) + f(\rho)) d   + \rho^2 f(\rho) + 3e^{\rhob} \rho d F + e^{\rhob}\rho^3 G \Big)\nonumber\\
   &\geq -e^\mrhob \Big(f''(\rho) d^2 + (2\rho f'(\rho) + f(\rho)) d   + \rho^2 f(\rho) + e^{mw(0)}\|G_0\|_{L^\infty}\rho^3 \Big)\nonumber\\
   &= -e^\mrhob f''(\rho) \big(d-\d_-(\rho)\big)\big(d-\d_+(\rho)\big),
\end{align}
where
\[
  \d_{\pm}(\rho) = \frac{2\rho f'(\rho) + f(\rho)\pm\sqrt{(2\rho f'(\rho) + f(\rho))^2-4f''(\rho)(\rho^2 f(\rho) + e^{mw(0)}\|G_0\|_{L^\infty}\rho^3)}}{-2f''(\rho)}.
\]
Here, we recall the assumption \eqref{eq:fluxcc} on $f$: $f''(\rho)<0$ for $\rho<\rhoc$.

For the case $\rho_0<\rhoc$, observe that $\d_\pm$ has a uniform lower bound, namely there exists a finite number $\d_-$, depending on $\rho_0$, such that
\[\d_\pm(\rho)\geq\d_->-\infty,\quad \forall~\rho\in[0,\rho_0].\]
Owing to \eqref{eq:rhobound}, we obtain that $\dot d(t)\geq0$ if $d(t)\leq\d_-$. We therefore have the lower bound
\[ d(t)\geq \min\{d_0,\underbar d\},\quad\forall~t\geq0.\]

For the case $\rho_0=\rhoc$, since the dynamics \eqref{eq:coupled} is locally well-posed, there exists a time $t_1>0$ such that the solution $(\rho(t),d(t))$ exists and is bounded for $t\in[0,t_1]$. From \eqref{eq:rhodecay} we know $\rho(t_1)<\rhoc$. Therefore, the dynamics starting from $t=t_1$ reduces to previous case, resulting a uniform-in-time lower bound.

For the asymptotic behavior of $d$, we apply the comparison principle \eqref{eq:comparison} and get:
\[\limsup_{t\to\infty}d(t)\leq\lim_{t\to\infty}\eta(\rho(t))=\eta(0)=0.\]
On the other hand, we observe that
\[\lim_{\rho\to0^+}\d_+(\rho)=\lim_{\rho\to0^+}\d_-(\rho)=0.\]
We argue from \eqref{eq:dlowerbound} that 
\[\liminf_{t\to\infty}d(t)\geq0.\]
Indeed, if $\liminf_{t\to\infty}d(t)=d_*<0$, we can find a time $t_*$ such that $\rho(t_*)$ is small enough so that $\d_-(\rho_*)>d_*/2$. Consequently, we have
\[\d_-(\rho(t))>\frac{d_*}{2},\quad\forall~t\geq t_*.\]
From \eqref{eq:dlowerbound}, we see that $d'(t)>0$ for any $t\geq t_*$, as long as $d(t)< d_*/2$. Hence, $\liminf_{t\to\infty}d(t)\geq d_*/2$. This leads to a contradiction.

We conclude with $d(t)\to0$ as $t\to\infty$.
\end{proof}

Finally, we apply Theorem \ref{thm:subcritical} along all characteristic paths to obtain a uniform bound on $\|\pa_x\rho\|_{L^\infty}$, which lead to the global well-posedness result in Theorem \ref{thm:main}.

\begin{proof}[Proof of Theorem \ref{thm:main}]
  For any $x\in\R$, from the subcritical initial condition \eqref{eq:subcritical}, we know $(\rho_0(x),\rho'_0(x))\in\Sigma$, where $\Sigma$ is defined in \eqref{eq:subinit}. Apply Theorem \ref{thm:subcritical}, we deduce that $\rho(t,X(t,x))$ and $\pa_x\rho(t,X(t,x))$ are bounded in all time, and the bounds are independent of the choice of $x$. Hence, there exists a constant $D>0$ such that 
  \[\|\pa_x\rho(t,\cdot)\|_{L^\infty}\leq D,\quad \forall~t\geq0.\]
  Consequently, the regularity criterion \eqref{eq:condition} holds for any finite time $T$. We apply Theorem \ref{thm:lwp} and conclude with the global regularity of the solution $(\rho,u)$.
\end{proof}

\appendix

\section{Energy estimates}
In this section, we prove Theorem \ref{thm:energy} by employing $\dot{H}^k$ energy estimates to the system:
\begin{equation}\label{eq:equiv}
  \begin{cases}
   \pa_tu+(u + U'(\rho)e^\mrhob \rho) \pa_x u = -U(\rho)e^\mrhob \displaystyle\int_0^\infty w'(z) \rho(x+z)(u(x+z)- u(x)) dz,\\    
  \pa_t\psi + u\pa_x\psi = 0,
  \end{cases}
\end{equation}

We begin by stating the following lemmas that will be utilized in the estimates. For the proof of the lemmas, we refer to e.g. \cite{kato1988commutator,li2019kato,lee2022sharp}.
\begin{lemma}[Fractional Leibniz rule]\label{lem:Leibniz}
Let $k\geq0,\,\, g,h \in L^\infty\cap \dot H^k(\R)$. There exists a constant $C>0$, depending only on $k$, such that
$$\|gh\|_{\dot H^k} \leq C\big( \|g\|_{L^\infty} \|h\|_{ \dot H^k} +
\|g\|_{\dot H^k} \|h\|_{L^\infty }\big) .$$
\end{lemma}

\begin{lemma}[Commutator estimate]\label{lem:commutator}
Let $k\geq1,\,\, g \in L^\infty\cap \dot H^k(\R)$, and $h\in L^\infty\cap \dot H^{k-1}(\R)$. There exists a constant $C>0$, depending only on $k$, such that
$$\|[\Lambda^k,g]h\|_{L^2} \leq C\big( \|\partial_x g\|_{L^\infty} \|h\|_{ \dot H^{k-1}} + \|g\|_{\dot H^k} \|h\|_{L^\infty }\big) ,$$
where the commutator is denoted by $[\Lambda^k,g]h=\Lambda^k(gh)-g\Lambda^kh$.

\end{lemma}

\begin{lemma}[Composition estimate]\label{lem:composition}
Let $k>0$, $g\in L^{\infty}\cap\dot H^k(\R)$, and $h\in C^{\lceil k \rceil}(\text{Range}(g))$. Then, the composition $h \circ g \in L^{\infty}\cap \dot H^k(\R)$. Moreover, there exists a constant $C>0$, depending on $k$, $\|h\|_{C^{\lceil k \rceil}(\text{Range}(g))}$, and $\|g\|_{L^\infty}$, such that 
$$\|h\circ g\|_{\dot H^k} \leq C\|g\|_{\dot H^k}.$$
\end{lemma}

We now proceed to analyze the evolution of the $\dot H^k$ energy of $u$ and $\psi$. 

We begin by acting $\Lambda^k$ on the $u$-equation of \eqref{eq:equiv} and integrating against $\Lambda^k u$:
\begin{align*}
\frac{1}{2}\frac{d}{dt} \|u\|_{\dot H^k}^2 =& -\int \Lambda^k u \cdot \Lambda^k\bigg( (u+U'(\rho)e^\mrhob\rho)\pa_x u\bigg) \,dx\\
&- \int_\R \Lambda^k \big(u(x)\big) \cdot \Lambda^k  \left( -U(\rho)e^\mrhob \int_0^\infty w'(z) \rho(x+z)(u(x+z)- u(x)) dz \right) dx \\
=:&\,\,{\rm I}+{\rm II}.
\end{align*}
For the first term ${\rm I}$, we further split it into two parts.
\[
{\rm I}= - \int (u+U'(\rho)e^\mrhob\rho) \cdot \Lambda^k u \cdot \Lambda^k \pa_x u\,dx - \int \Lambda^k u \cdot [\Lambda^k,(u+U'(\rho)e^\mrhob\rho)]\pa_x u \,dx
=:{\rm I}_1+{\rm I}_2.
\]
The term ${\rm I}_1$ contains a dangerous part $\Lambda^k \pa_x u$, where we apply integration by parts and \eqref{eq:nonlocalpax} to obtain:
\begin{align}
{\rm I}_1 &= - \int (u+U'(\rho)e^\mrhob\rho) \cdot \Lambda^k u \cdot \Lambda^k \pa_x u\,dx= -\int  (u+U'(\rho)e^\mrhob\rho) \cdot \Lambda^k u \cdot \pa_x\big(\Lambda^k u\big)\,dx\nonumber\\
&= -\int  (u+U'(\rho)e^\mrhob\rho)  \cdot \pa_x\left(\frac{(\Lambda^k u)^2}{2}\right)\,dx= \int  \pa_x(u+U'(\rho)e^\mrhob\rho)   \cdot \left(\frac{(\Lambda^k u)^2}{2}\right)\,dx\nonumber\\
&\leq  C\big(1+\|\pa_x \rho\|_{L^\infty}+\|\pa_x u\|_{L^\infty} \big) \|u\|_{\dot H^k}^2,\label{eq:I1}
\end{align}
where
$$C = C\Big( \|U\|_{C^2([0,\rho_m])}, w(0) \Big).$$
For the remaining commutator ${\rm I}_2$, we apply H\"older's inequality and Lemma \ref{lem:commutator} and get:
\begin{align*}
{\rm I}_2 &=  - \int \Lambda^k u \cdot [\Lambda^k,(u+U'(\rho)e^\mrhob\rho)]\pa_x u \,dx \leq \|u\|_{\dot H^k}\|[\Lambda^k,(u+U'(\rho)e^\mrhob\rho)]\pa_x u \|_{L^2}\\
&\leq C\|u\|_{\dot H^k} \Big( \|\partial_x (u+U'(\rho)e^\mrhob\rho)\|_{L^\infty} \|\pa_x u\|_{ \dot H^{k-1}} + \|(u+U'(\rho)e^\mrhob\rho)\|_{\dot H^k} \|\pa_x u\|_{L^\infty }   \Big).
\end{align*}
To control $\|\partial_x (U'(\rho)e^\mrhob\rho)\|_{L^\infty}$, we apply \eqref{eq:nonlocalbound} and \eqref{eq:nonlocalpax} to obtain:
\begin{align*}
\|\partial_x (U'(\rho)e^\mrhob\rho)\|_{L^\infty}&\leq \|U'(\rho)+\rho U''(\rho)\|_{L^\infty}\|\pa_x\rho\|_{L^\infty}\|e^\mrhob\|_{L^\infty}+\|\rho U'(\rho)\|_{L^\infty}\|\pa_x(e^\mrhob)\|_{L^\infty}\\
&\leq C\|U\|_{C^2([0,\rhoM])}(w(0)+\|\pa_x\rho\|_{L^\infty}).
\end{align*}
For $\|U'(\rho)e^\mrhob\rho\|_{\dot H^k}$ we use Lemmas \ref{lem:Leibniz} and \ref{lem:composition}:
\begin{align*}
\|U'(\rho)e^\mrhob\rho\|_{\dot H^k} &\leq C\big( \|\rho U'(\rho)\|_{L^\infty} \|e^\mrhob\|_{ \dot H^k} +\|\rho U'(\rho)\|_{\dot H^k} \|e^\mrhob\|_{L^\infty }\big) \\
&\leq C\|U\|_{C^1([0,\rhoM])} \|e^\mrhob\|_{ \dot H^k} + C(\|U\|_{C^{\lceil k \rceil+1}([0,\rho_M])})\|\rho\|_{\dot H^k}.
\end{align*}
The term $\|e^\mrhob\|_{\dot H^k}$ can be estimated by applying Lemma \ref{lem:composition} with
$g(x)=\rhob(t,x)$ and  $h(x)=e^{-x}$. From \eqref{eq:rhobbound} we know
$\|g\|_{L^\infty}\leq  mw(0)$. Moreover,
$h \in C^\infty([0,  mw(0)])$. Therefore, we have
\[
    \|e^{-\rhob}\|_{\dot H^k} \leq C(k,  mw(0))\|\rhob\|_{\dot H^k},
\]
Putting everything together, we have the following estimate on ${\rm I}_2$:
\begin{equation}\label{eq:I2}
 {\rm I}_2	\leq C \big( 1+\|\pa_x \rho\|_{L^\infty}+\|\pa_x u\|_{L^\infty}  \big)\big(\|u\|_{\dot H^k}^2 + \|\rho\|^2_{\dot{H}^k} + \|\rhob\|^2_{\dot{H}^k}\big),
\end{equation}
where
\[C = C\Big(k, m, \|U\|_{C^{\max\{2,\lceil k \rceil+1\}}([0,\rho_M])}, w(0) \Big).\]

Next, we proceed to the term ${\rm II}$. Apply Leibniz rule (Lemma \ref{lem:Leibniz}): 
\begin{align*}
 {\rm II} &\leq  \|u\|_{\dot H^k} \left\|U(\rho)e^\mrhob \int_0^\infty w'(z) \rho(\cdot+z)(u(\cdot+z)- u(\cdot)) dz  \right\|_{\dot H^k}\\
 &\leq C\|u\|_{\dot H^k}\bigg(\|U(\rho)e^\mrhob\|_{\dot H^k} \left\|\int_0^\infty w'(z) \rho(\cdot+z)(u(\cdot+z)- u(\cdot)) dz  \right\|_{L^\infty}\bigg.\\
 &\qquad\qquad \bigg.+\|U(\rho)e^\mrhob\|_{L^\infty} \left\|\int_0^\infty w'(z) \rho(\cdot+z)(u(\cdot+z)- u(\cdot)) dz  \right\|_{\dot H^k}\bigg).
\end{align*}
Let us further estimate each term on the right-hand side of the inequality above, continuing to use Leibniz rule and the composition estimate (Lemma \ref{lem:composition}) if needed.
\begin{align*}
 \|U(\rho)e^\mrhob\|_{L^\infty}&=\|U(\rho)\|_{L^\infty}\|e^\mrhob\|_{L^\infty}\leq 1,\\
 \|U(\rho)e^\mrhob\|_{\dot H^k}&\leq C (\|U(\rho)\|_{\dot H^k}\|e^\mrhob\|_{L^\infty}+\|U(\rho)\|_{L^\infty}\|e^\mrhob\|_{\dot H^k})\\
 &\leq C(\|U\|_{C^{\lceil k\rceil+1}([0,\rhoM])})\|\rho\|_{\dot H^k} + C(k,  mw(0))\|\rhob\|_{\dot H^k}.
\end{align*}
For the remaining terms, we make use of the fact that $w'(z)\leq0$. The $L^\infty$ bound is given by:
\begin{align*}
 &\left|\int_0^\infty w'(z) \rho(x+z)(u(x+z)- u(x)) dz\right|\\
 &\leq \int_0^\infty (-w'(z)) \rho(x+z) |u(x+z)- u(x)| dz\leq  \int_0^\infty (-w'(z)) \cdot 1\cdot 2\, dz=2w(0).
\end{align*}
For the $\dot H^k$ bound, we apply Minkowski integral inequality and get:
\begin{align*}
&\left\|\int_0^\infty w'(z) \rho(\cdot+z)(u(\cdot+z)- u(\cdot)) dz  \right\|_{\dot H^k}\\
&\leq \int_0^\infty (-w'(z))\Big\|\rho(\cdot+z)(u(\cdot+z)- u(\cdot))\Big\|_{\dot H^k}dz\\
&\leq 2w(0)\cdot C\big(\|\rho\|_{\dot H^k}\|u\|_{L^\infty}+\|\rho\|_{L^\infty}\|u\|_{\dot H^k}\big)\leq C w(0)(\|u\|_{\dot H^k}  + \|\rho\|_{\dot H^k}).
\end{align*}
Overall, the term {\rm II} can be estimated by
\begin{equation}\label{eq:II}	
 {\rm II} \leq C \big(\|u\|_{\dot H^k}^2+\|\rho\|_{\dot H^k}^2+\|\rhob\|_{\dot H^k}^2\big),
\end{equation}
where
\[
C = C\big(k,m, \|U\|_{C^{\lceil k\rceil+1}([0,\rhoM])}, w(0)\big).
\]

Collecting the estimates in \eqref{eq:I1}, \eqref{eq:I2} and \eqref{eq:II}, we deduce the following control of the growth of $\|u\|_{\dot H^k}$:
\begin{equation}\label{eq:e1est}
\frac{1}{2}\frac{d}{dt} \|u\|_{\dot H^k}^2 \leq C \big(1+\|\pa_x \rho\|_{L^\infty} +\|\pa_x u\|_{L^\infty}  \big)\big(\|u\|_{\dot H^k}^2+\|\rho\|_{\dot H^k}^2+\|\rhob\|_{\dot H^k}^2\big),
\end{equation}
where the constant
\[
C=C\big(k,  m, \|U\|_{C^{\max\{2,\lceil k\rceil +1\}}([0,\rho_M])}, w(0) \big).
\]

Next, we proceed similarly with the $\psi$-equation in \eqref{eq:equiv}, acting $\Lambda^k$ and integrating against $\Lambda^k \psi$:
\begin{align*}
\frac{1}{2}\frac{d}{dt} \|\psi\|_{\dot H^k}^2 &= - \int \Lambda^k \psi \cdot\Lambda^k\big(u\pa_x\psi\big) \,dx \\ 
&= -\int u \cdot \Lambda^k\psi\cdot \Lambda^k\pa_x\psi \,dx - \int\Lambda^k\psi \cdot[\Lambda^k, u]\pa_x\psi \,dx =: {\rm III}_1 + {\rm III}_2,
\end{align*}
where we split the right-hand side into two terms (like we did for ${\rm I}$).

We again proceed with each term individually. First, use integration by parts to take care of the dangerous term $\Lambda^k\pa_x\psi$ in ${\rm III}_1$:
\[
{\rm III}_1 = -\int u \cdot \pa_x\left(\frac{\Lambda^k\psi^2}{2}\right)\,dx = \int \pa_x u \cdot \frac{\Lambda^k\psi^2}{2} \,dx\leq \frac12\|\pa_x u\|_{L^\infty}\|\psi\|_{\dot H^k}^2.
\]
Then, for ${\rm III}_2$ we again apply H\"older's inequality and the commutator estimate (Lemma \ref{lem:commutator}):
\begin{align*}
{\rm III}_2 &\leq \|\psi\|_{\dot H^k} \|[\Lambda^k, u]\pa_x\psi\|_{L^2}\leq C\|\psi\|_{\dot H^k} \Big(  \|\partial_x u\|_{L^\infty} \|\pa_x\psi\|_{ \dot H^{k-1}} + \|u\|_{\dot H^k} \|\pa_x\psi\|_{L^\infty } \Big)\\
 & \leq C\Big(\|\pa_x u\|_{L^\infty}+ \|\pa_x\psi\|_{L^\infty} \Big) \big(\|u\|_{\dot H^k}^2+\|\psi\|_{\dot H^k}^2 \big).
\end{align*}
Combining the two estimates yield:
\begin{align}\label{eq:e2est}
    \frac{1}{2}\frac{d}{dt} \|\psi\|_{\dot H^k}^2 \leq C\big(\|\pa_x u\|_{L^\infty} +\|\pa_x \psi\|_{L^\infty} \big)\big(\|u\|_{\dot H^k}^2+\|\psi\|_{\dot H^k}^2\big).
\end{align}

Finally, we put together \eqref{eq:e1est} and \eqref{eq:e2est} and conclude with the energy estimate
\begin{equation}\label{eq:Eest}
E_k'(t) \leq C\big(\|\pa_x \rho\|_{L^\infty} + \|\pa_x u\|_{L^\infty} + \|\pa_x \psi\|_{L^\infty} \big)\big(E_k+\|\rho\|_{\dot H^k}^2+\|\rhob\|_{\dot H^k}^2).
\end{equation}

The next Proposition allows us to absorb $\|\rho\|_{\dot H^k}^2$ by $E_k$ and $\|\rhob\|_{\dot H^k}^2$.

\begin{proposition}\label{prop:rhoHk}
Let $k\geq1$. Then there exists a constant $C>0$, depending on $k, \|\psi_0\|_{L^\infty}, \|U^{-1}\|_{C^{\lceil k\rceil}([U(\rhoM),1])}$ and $w(0)$, such that 
\begin{equation}\label{eq:rhoHk}
\|\rho\|_{\dot H^k} \leq C\big(\|u\|_{\dot H^k}+\|\psi\|_{\dot H^k}+\|\rhob\|_{\dot H^k}\big).
\end{equation}
\end{proposition}
\begin{proof}
  Beginning with the relationship 
  \[u =\psi+U(\rho)e^\mrhob,\]
  we may write
  \[\rho = U^{-1} \Big( (u-\psi) e^{\rhob}  \Big)\]
  since $U$ is monotone decreasing and hence invertible. We therefore have in light of the composition estimate (Lemma \ref{lem:composition}):
  \[
  \|\rho\|_{\dot H^k} = \left\| U^{-1} \Big( (u-\psi) e^{\rhob}  \Big)\right\|_{\dot H^k} \leq C\| (u-\psi) e^{\rhob} \|_{\dot H^k},
  \]
  where the constant $C$ depends on $\|(u-\psi) e^{\rhob}\|_{L^\infty}$ and $\|U^{-1}\|_{C^{\lceil k\rceil}(Range((u-\psi) e^{\rhob}))}$. Since $(u-\psi) e^{\rhob}=U(\rho)$, we have $\|(u-\psi) e^{\rhob}\|_{L^\infty}\leq1$. The range of $U(\rho)$, under assumption \eqref{A2} or \eqref{A2p}, is $[0,1]$ or $[U(\rhoM),1]$, respectively. Therefore, we have
  \[C=C\big(\|U^{-1}\|_{C^{\lceil k\rceil}([U(\rhoM),1])}\big).\]
  
  We continue the estimate using the Leibniz rule (Lemma \ref{lem:Leibniz}) and obtain:
  \begin{align*}
        \|\rho\|_{\dot H^k} &\leq C\big( \|u-\psi\|_{L^\infty}\|e^{\rhob} \|_{\dot H^k} +\|u-\psi\|_{\dot H^k}\|e^{\rhob} \|_{L^\infty}    \big)
        \leq C\big( \|u \|_{\dot H^k} +\|\psi\|_{\dot H^k} +  \|\rhob\|_{\dot H^k}  \big),
    \end{align*}
  where
    $$C=C\big(k, \|\psi_0\|_{L^\infty}, \|U^{-1}\|_{C^{\lceil k\rceil}([U(\rhoM),1])}, w(0)\big).$$
\end{proof}

Next, we estimate the nonlocal term $\|\rhob\|_{\dot H^k}$ by $\|\rho\|_{\dot H^{k-1}}$. For any $k\geq1$, we have
\begin{align}
   \|\rhob\|_{\dot H^k} &= \|\pa_x\Lambda^{k-1} \rhob \|_{L^2} = \left \|\int_0^\infty w(z)\pa_x\Lambda^{k-1} \rho(\cdot+z) dz \right \|_{L^2}\nonumber\\
   &= \left \|-w(0)\Lambda^{k-1}\rho-\int_0^\infty w'(z)\Lambda^{k-1} \rho(\cdot+z) dz \right \|_{L^2}\nonumber\\
   &\leq w(0)\|\rho\|_{\dot H^{k-1}}+\int_0^\infty (-w'(z))\big\|\Lambda^{k-1}\rho(\cdot+z)\big\|_{L^2}dz=2w(0)\|\rho\|_{\dot H^{k-1}}\label{eq:rhobHk}.
\end{align}

Applying \eqref{eq:rhobHk} to \eqref{eq:rhoHk}, we arrive at the following bound:
\[
\|\rho\|_{\dot H^k} \leq C_0\big(\|u\|_{\dot H^k}+\|\psi\|_{\dot H^k}+\|\rho\|_{\dot H^{k-1}}\big).
\]
Together with the interpolation inequality
\[
 \|\rho\|_{\dot H^{k-1}}\leq \frac{1}{2C_0}\|\rho\|_{\dot H^k} + C\|\rho\|_{L^2},
\]
we obtain the estimate
\begin{equation}\label{eq:rhoHk2}
\|\rho\|_{\dot H^k}\leq C\big(\|u\|_{\dot H^k}+\|\psi\|_{\dot H^k}+\|\rho\|_{L^2}\big),	
\end{equation}
for any $k\geq1$. Consequently, the last two terms in \eqref{eq:Eest} can be controlled by
\begin{align*}
 \|\rho\|_{\dot H^k}^2+\|\rhob\|_{\dot H^k}^2 &\leq C\big(\|\rho\|_{\dot H^k}^2+\|\rho\|_{\dot H^{k-1}}^2\big)\leq C\big(\|\rho\|_{\dot H^k}^2+\|\rho\|_{L^2}^2\big)\\
 &\leq C\big(\|u\|_{\dot H^k}^2+\|\psi\|_{\dot H^k}^2+\|\rho\|_{L^2}^2\big).
\end{align*}
We conclude with the estimate \eqref{eq:energy}. This finishes the proof of Theorem \ref{thm:energy}.

\bibliographystyle{abbrv}
\bibliography{traffic}

\end{document}